\documentclass[12pt]{amsart}

\usepackage{etex}
\usepackage{amscd}
\usepackage{amsmath}
\usepackage{amsfonts}
\usepackage{amssymb}
\usepackage{graphics}
\usepackage{epsfig}
\usepackage{pictex} 
\usepackage{epstopdf}

\font\Bbb=msbm10 scaled \magstep 2

\def\C{\hbox{\Bbb C}}
\def\R{\hbox{\Bbb R}}
\def\Z{\hbox{\Bbb Z}}
\def\N{\hbox{\Bbb N}}

\font\midBbb=msbm8

\def\midC{\hbox{\midBbb C}}
\def\midR{\hbox{\midBbb R}}
\def\midZ{\hbox{\midBbb Z}}
\def\midN{\hbox{\midBbb N}}

\def\Log{\hbox{\rm Log}}

\newtheorem{theorem}{\bf Theorem}
\newtheorem{lemma}[theorem]{\bf Lemma}
\newtheorem{remark}[theorem]{\bf Remark}
\newtheorem{example}[theorem]{\bf Example}
\newtheorem{proposition}[theorem]{\bf Proposition}
\newtheorem{corollary}[theorem]{\bf Corollary}
\newtheorem{definition}[theorem]{\bf Definition}

\numberwithin{equation}{section} \numberwithin{theorem}{section}

\title{Maximally reducible monodromy of bivariate hypergeometric systems}

\author{Timur Sadykov}
\address{Department of Mathematics
\newline \indent and Computer Science,
\newline \indent Russian State Plekhanov University
\newline \indent 125993, Moscow, Russia.}
\email{SadykovTM@rsute.ru}
\thanks{The first author was supported by the grant of the Government
of the Russian Federation for investigations under the guidance of
the leading scientists of the Siberian Federal University
(contract No.~14.Y26.31.0006), by the grants of the Russian
Foundation for Basic Research (grants no.~13-01-12417-ofi-m2,
15-31-20008-mol-a-ved), as well as by the Japanese Society for the
Promotion of Science. The second author was supported by JSPS
grant no.~20540086.}

\author{Susumu Tanab\'e}
\address{Department of Mathematics,
\newline \indent Galatasaray University,
\newline \indent 34357, Istanbul, Turkey.}
\email{tanabe@gsu.edu.tr}

\begin{document}

\begin{abstract}
We investigate branching of solutions to holonomic bivariate
hypergeometric systems of Horn's type. Special attention is paid
to the invariant subspace of Puiseux polynomial solutions. We
mainly study Horn systems defined by simplicial configurations and
Horn systems whose Ore-Sato polygons are either zonotopes or
Minkowski sums of a triangle and segments proportional to its
sides. We prove a necessary and sufficient condition for the
monodromy representation to be maximally reducible, that is, for
the space of holomorphic solutions to split into the direct sum of
one-dimensional invariant subspaces.

\end{abstract}

\maketitle

\section{Introduction
\label{sec:introduction}}
To compute the monodromy group of a differential equation or a system of such equations
is a notoriously difficult problem in the analytic theory of differential equations.
One of the reasons for this is that the computation of the monodromy group requires
full understanding of the structure of the solution space of the system of
differential equations under study, including the dimension of this space,
a basis in it, the fundamental group of the complement to singularities of the system as well as analytic continuation
and branching properties of the chosen basis.

The purpose of the present paper is to investigate the monodromy
of certain families of systems of partial differential equations
of hypergeometric type. It uses and extends the results
in~\cite{Sadykov-Doklady2008} and~\cite{Sadykov-Journal of SFU}.
While the monodromy group of the classical Gauss second-order
hypergeometric differential equation has been computed by Schwarz
and the monodromy of the ordinary generalised hypergeometric
equation has been described in~\cite{BeukersHeckman}, the problem
of finding the monodromy group of a general hypergeometric system
of partial differential equations remains unsolved despite all the
effort and several well-understood special cases
(see~\cite{Beukers2009}, \cite{Beukers2011} and the references
therein). The original motivation for the results presented in the
paper goes back to the work~\cite{CDS} where the authors have
posed the problem of describing the Gelfand-Kapranov-Zelevinsky
(GKZ) nonconfluent hypergeometric systems (see~\cite{gkz89}),
whose solution space contains a nonzero rational function for a
suitable choice of its parameters. In terms of monodromy, this is
equivalent to the existence of a one-dimensional subspace in the
space of holomorphic solutions to the system under study with the
trivial action of monodromy on it.

In the present paper, we solve a closely related problem of
describing all holonomic bivariate hypergeometric systems in the
sense of Horn (see~\cite{DMS} and the references therein) whose
solution space splits into a direct sum of one-dimensional
monodromy invariant subspaces
(Theorem~\ref{thm:zonotopesHavePolBasis}). We call such a
monodromy representation {\it maximally reducible.} The relation
between GKZ and Horn hypergeometric systems has been studied in
detail in Section~5 of~\cite{DMS}: for any GKZ system there exists
a canonically defined Horn system and a naturally defined
bijective map from a subspace in the space of its analytic
solutions into the space of solutions to the GKZ system. The
solutions of the Horn system that are not taken into account by
this map are its persistent Puiseux polynomial solutions in the
sense of Definition~\ref{def:PersistentPolynomialSolution} below.
Here and throughout the paper by a Puiseux polynomial we mean a
finite linear combination of monomials with (in general) arbitrary
complex exponents. As it has been announced in Theorem 5.3
of~\cite{DMS}, persistent polynomial solutions are the cokernel of
the map from GKZ solutions to Horn system solutions.

In our formulation, the above mentioned question of ~\cite{CDS}
can be answered in the following manner. The dimension of the
space of non-persistent Puiseux polynomial solutions to a Horn
system is equal to that of the space of Puiseux polynomial
solutions to the corresponding GKZ system. For the bivariate Horn
system, full characterisation of persistent solutions is given in
Proposition~\ref{prop:persistentPolSol} and
Corollary~\ref{cor:persistentPolsol}.

The authors are thankful to the referee for the careful reading of
the manuscript and numerous suggestions that have led to a
substantial improvement of the paper. Publication of the paper in
the present special issue of the journal is a tribute to
A.\,A.~Bolibrukh for his constant support of the second author
(S.T.) over many years.


\section{Notation, definitions and preliminaries
\label{sec:notation}}

Throughout the paper, the following notation will be used:

$n=$ the number of $x$ variables;

$m=$ the number of rows in the matrix defining the Horn system;

$\nu(a_1, b_1; a_2, b_2)\equiv \nu \left(
\begin{array}{rr}
 a_1 & b_1 \\
 a_2 & b_2 \\
\end{array}
\right) =$ the index of the two vectors $(a_1, b_1),(a_2, b_2),$
see Definition~\ref{def:index};

for $m=(m_1,\ldots,m_n),$ $|m|=\sum_{i=1}^n m_i$ and $m!=m_{1}!
\ldots m_{n}!;$

for $x=(x_1,\ldots,x_n)$ and $m=(m_1,\ldots,m_n),$
$x^m=x_{1}^{m_1}\ldots x_{n}^{m_n};$

$\Z_{\geq 0}=$\,the set of non-negative integers, $\Z_{\leq
0}=$\,the set of non-positive integers;

${\rm Horn}(\varphi)=$ the Horn hypergeometric system defined by
the Ore-Sato coefficient $\varphi,$ see Definition~\ref{horn}.

${\rm Horn}(A,c)=$ the Horn hypergeometric system defined by the
Ore-Sato coefficient~(\ref{oresatocoeff}) with $t_i=1$ for any
$i=1,\ldots,n$ and $U(s)\equiv 1.$ See the construction after
Definition~\ref{horn};

$\Psi(\varphi)=$ the subspace of Puiseux polynomial solutions to the
Horn system defined by the Ore-Sato coefficient $\varphi,$ see Definition~\ref{def:Horn system};

$\Psi_{0}(\varphi) \subset \Psi(\varphi)$ is the subspace of persistent Puiseux polynomial solutions to the
Horn system defined by the Ore-Sato coefficient $\varphi$, see Definition~\ref{def:PersistentPolynomialSolution};

$\mathcal{F}=$ the set of all pure fully supported solutions to a Horn
system. Observe that it is in general not a linear subspace
since the intersection of the domains of convergence of all elements
in~$\mathcal{F}$ may be empty;

$\mathcal{F}_{x^{(0)}}=$ the linear space of fully supported solutions
to a Horn system which converge at a nonsingular point~$x^{(0)};$

$\mathcal{A}(\varphi)=$ the amoeba of the singularity of an Ore-Sato coefficient~$\varphi;$ see Definition~\ref{def:amoeba};

$C^{\vee}=$ the dual of a convex cone~$C;$

$\mathcal{P}(\varphi)$ is the polygon of the Ore-Sato coefficient $\varphi,$ see Definition~\ref{def:polygon}.

For an Ore-Sato coefficient~$\varphi$ and $\zeta\in\R^n$ we set
$$
M(\varphi,\zeta) =\left\{\begin{array}{ll}
\mbox{the connected component of\ } ^c\!\mathcal{A}(\varphi)
\mbox{\ which contains\ }~\zeta, \mbox{if\ } \zeta\in {^c\!\mathcal{A}}(\varphi), \\
\R^n, \mbox{\ if\ } \zeta\in\mathcal{A}(\varphi);
\end{array}
\right.
$$

$S({\operatorname{Horn}}(A,c))$ is the space of solutions to the
system $\operatorname{Horn}(A,c)$, that are holomorphic away from
the singular hypersurface.

\vskip0.5cm

\begin{definition}
\rm A formal Laurent series
\begin{equation}
\sum_{s\in\midZ^n} \varphi(s) \, x^{s}
\label{series}
\end{equation}
is called {\it hypergeometric} if for any $j=1,\ldots,n$ the
quotient $\varphi(s+e_{j})/\varphi(s)$ is a rational function
in~$s = (s_1,\ldots,s_n).$ Throughout the paper we denote this rational function by
$P_{j}(s)/Q_{j}(s+e_{j}).$ Here~${\{e_j\}}_{j=1}^{n}$ is the
standard basis of the lattice~$\Z^n.$ By the {\it support} of this
series we mean the subset of~$\Z^n$ on which $\varphi(s)\neq 0.$
We say that such a series is {\it fully supported,} if the convex
hull of its support contains (a translation of) an open
$n$-dimensional cone.
\end{definition}

A {\it hypergeometric function} is a (multi-valued) analytic function
obtained by means of analytic continuation of a hypergeometric series
with a nonempty domain of convergence along all possible paths.

\begin{theorem} {\rm (Ore, Sato \cite{GGR}, \cite{Sato})} The coefficients of
a hypergeometric series are given by the formula
\begin{equation}
\varphi(s) = t^{s} \, U(s) \, \prod_{i=1}^{m} \Gamma(\langle {\bf A}_{i}, s
\rangle + c_{i}),
\label{oresatocoeff}
\end{equation}
where $t^s = t_{1}^{s_1}\ldots t_{n}^{s_n},$ $t_i, c_i\in\C,$
${\bf A}_i=(A_{i,1}, \ldots A_{i,n})$ $\in\Z^n,$ $i =1, \ldots, m,$ and
$U(s)$~is a product of  certain rational function and a periodic
function $\phi(s)$ s.t. $\phi(s+e_j) =\phi(s)$ for every $j =
1,\ldots,n$. \label{thm:Ore-Sato}
\end{theorem}

In the article ~\cite{Sato} Appendix (A.3) a precise description of rational
function factor of $U(s)$ is available.

We will call any function of the form~(\ref{oresatocoeff}) {\it
the Ore-Sato coefficient of a hypergeometric series.}  We remark that
in view of the formula
$$
\sin(\pi z)\Gamma(1-z)\Gamma(z)=\pi
$$
an Ore-Sato coefficient can be a function of the form
$$
\varphi(s)=t^{s}\prod_{i\in{\mathbf I}}
\Gamma(\langle{\mathbf A}_{i},s\rangle+c_{i})
\prod_{j\notin{\mathbf I}}
\frac{e^{\pi\sqrt{-1}(\langle{\mathbf A}_{j},s\rangle+c_{j})}}
{\Gamma(1-\langle{\mathbf A}_{j},s\rangle-c_{j})},
$$
where ${\mathbf I}\subset\{1,\dots,m\}$.

Given the above data ($t_i, c_i, {\bf A}_i, U(s)$) that determines
the coefficient of a hypergeometric series, it is straightforward
to compute the rational functions $P_{i}(s)/Q_{i}(s+e_{i})$ using
the $\Gamma$-function identity. The converse requires solving a
system of difference equations which is only solvable under some
compatibility conditions on $P_i,Q_i.$ A careful analysis of this
system of difference equations has been performed
in~\cite{Sadykov-SMZh}.

In this
paper the Ore-Sato coefficient~(\ref{oresatocoeff}) plays the role
of a primary object which generates everything else: the series,
the system of differential equations, the algebraic hypersurface
containing the singularities of its solutions, the amoeba of its
defining polynomial, and, ultimately, the monodromy group of the
hypergeometric system of differential equations. We will also
assume that $m \geq n$ since otherwise the corresponding
hypergeometric series~(\ref{series}) is just a linear combination
of hypergeometric series in fewer variables (times arbitrary
function in remaining variables that makes the system
non-holonomic) and~$n$ can be reduced to meet the inequality.

\begin{definition}\label{def:Horn system}{\rm
{\it The Horn system of an Ore-Sato coefficient.} A (formal)
Laurent series $\sum_{s\in\midZ^n}\varphi(s)x^s$ whose coefficient
satisfies the relations $\varphi(s+e_j)/\varphi(s) =
P_j(s)/Q_j(s+e_j)$ is a (formal) solution to the following system
of partial differential equations of hypergeometric type
\begin{equation}
x_j P_j(\theta)f(x) = Q_j(\theta)f(x), \,\,\, j=1,\ldots,n.
\label{horn}
\end{equation}
Here $\theta=(\theta_1,\ldots,\theta_n),$ $\theta_j =
x_j\frac{\partial}{\partial x_j}.$ The system~(\ref{horn}) will be
referred to as {\it the Horn hypergeometric system defined by the
Ore-Sato coefficient~$\varphi(s)$} (see~\cite{GGR}) and denoted by
${\rm Horn}(\varphi)$.  We shall denote by  $S({\rm
Horn}(\varphi))$ the solution space to  ${\rm Horn}(\varphi)$. In
this paper we treat only holonomic Horn hypergeometric systems if
not otherwise specified i.e. $ {\rm rank} ({\rm Horn}(\varphi))$
is always assumed to be finite. A necessary and sufficient
condition for a system ${\rm Horn}(\varphi)$ to be holonomic has
been established
in~\cite{DMM}, Theorem 6.3.   }
\end{definition}
We will often be dealing with the important special case of an
Ore-Sato coefficient~(\ref{oresatocoeff}) where $t_i=1$ for any
$i=1,\ldots,n$ and $U(s)\equiv 1.$ The Horn system associated with
such an Ore-Sato coefficient will be denoted by ${\rm Horn}(A,c),$
where~$A$ is the matrix with the rows ${\bf A}_1,\ldots, {\bf A}_m \in \Z^n$
and $c=(c_1, \ldots, c_m)\in\C^m.$ In this case the following
operators $P_j(\theta)$ and $Q_j(\theta)$ explicitly determine the
system~(\ref{horn}):
$$ P_j(s) = \prod_{i : A_{i,j}>0} \prod_{\ell_j^{(i)}=0}^{A_{i,j}-1} \left( \langle {\bf A}_{i}, s
\rangle + c_{i} + \ell_j^{(i)}\right), $$
$$ Q_j(s) = \prod_{i : A_{i,j}<0} \prod_{\ell_j^{(i)}=0}^{|A_{i,j}| -1} \left( \langle {\bf A}_{i}, s
\rangle + c_{i} + \ell_j^{(i)}\right). $$

\begin{definition}
\rm The Ore-Sato coefficient~(\ref{oresatocoeff}), the
corresponding hypergeometric series~(\ref{series}), and the
associated hypergeometric system~(\ref{horn}) are called {\it
nonconfluent} if
\begin{equation}
\sum_{i=1}^{m} {\bf A}_{i} = 0.
\label{nonconfluency}
\end{equation}
\end{definition}
It is a well known fact (e.g. \cite{DMM}, Theorem 6.3) that a nonconfluent holonomic hypergeometric system
is a regular holonomic system i.e. every solution admits polynomial growth when approaching its singular loci.

\begin{definition}\label{def:polygon}
\rm {\it The polygon of a nonconfluent Ore-Sato coefficient in two
variables.} Using, if necessary, the Gauss multiplication formula
for the $\Gamma$-function and $N \in \N$,
$$
\Gamma(\langle {\bf A}_i, s \rangle + c_i) =
$$
$$
\frac{N^{ \langle {\bf A}_i,s \rangle + c_i}}{(2 \pi)^{(N-1)/2} \sqrt N} \,  \Gamma \left( \frac{\langle {\bf A}_i,s \rangle + c_i}{N} \right)
\Gamma\left( \frac{\langle {\bf A}_i,s \rangle + c_i + 1}{N} \right) \ldots \Gamma \left( \frac{ \langle {\bf A}_i,s \rangle + c_i + N - 1}{N} \right),
$$
we may without loss of generality assume that for any
$i=1,\ldots,p$ the nonzero components of the vector~${\bf A}_{i}$ are
relatively prime. Let~$l_{i}$ denote the generator of the
sublattice $\{ s\in\Z^{2} : \langle {\bf A}_{i}, s \rangle = 0\}$ and
let~$k_{i}$ be the number of elements in the multiset
$\{{\bf A}_{1},\ldots,{\bf A}_{m}\}$ which coincide with~${\bf A}_{i}.$ The
nonconfluency condition~(\ref{nonconfluency}) implies that there
exists a uniquely determined (up to a translation) integer convex
polygon whose sides are translations of the vectors $k_{i}l_{i},$
the vectors ${\bf A}_{1},\ldots,{\bf A}_{m}$ being the outer normals to its
sides. The number of sides of this polygon coincides with the
number of different elements in the multiset of vectors
$\{{\bf A}_{1},\ldots,{\bf A}_{m}\}.$ We call this polygon {\it the polygon
of the Ore-Sato coefficient~(\ref{oresatocoeff})} and denote it
by~$\mathcal{P}({\varphi}).$
\end{definition}

Conversely, any convex integer polygon determines a
$(m\times2)$-matrix whose rows sum up to the zero vector and
therefore (together with a vector of parameters) a nonconfluent
hypergeometric system of equations. We will denote this system by
${\operatorname{Horn}}(A(\mathcal{P}),c)$. This relation is
illustrated by example~\ref{e4.5}.


\begin{definition}
\label{def:index}
{\rm For a pair of vectors
$(a_1,b_1),$ $(a_2,b_2)\in\Z^2$ we set
$$
\nu(a_1,b_1;a_2,b_2) = \left\{ \begin{array}{ll} \min(|a_1b_2|,|b_1a_2|),
\quad  &
\mbox{if~$(a_1,b_1)$,~$(a_2,b_2)$ are}\\
      &\mbox{in opposite open quadrants of $\Z^2$},\\
0, & \mbox{otherwise}. \end{array} \right.
$$
The number~$\nu(a_1,b_1;a_2,b_2)$ is called the {\em index} associated
with the lattice vectors~$(a_1,b_1)$ and~$(a_2,b_2)$. The index of the
rows of a $2\times 2$ matrix~$M$ will be denoted by $\nu(M).$}
\end{definition}
\begin{definition} \rm
By the {\it initial exponent} of a multiple hypergeometric series
$$
x^{\alpha} \sum\limits_{s\in \midZ^n}
\varphi(s) \, x^{s}
$$
we mean the vector $\alpha=(\alpha_1,\ldots,\alpha_n)\in\C^n.$
Observe that the initial exponent of such a series is only defined
up to shifts by integer vectors.
However, in the view of Proposition~\ref{prop:intertwiningOperators} and
Corollary~\ref{cor:equivalentRepresentations} (to be proved in Section~\ref{sec:atomicHorn})
this is exactly what we need for computing monodromy of hypergeometric systems.
\label{def:initialExponent}
\end{definition}
\begin{definition}
\rm The support of a series solution to~(\ref{horn}) is called
{\it irreducible} if there exists no series solution to~(\ref{horn})
supported in its proper nonempty subset.
\label{def:irreducibleSupport}
\end{definition}
\begin{definition}
\rm A series solution with irreducible support $f(x) =
\sum_{\alpha\in \Lambda} c_\alpha x^{\alpha}$ to a Horn system is
called {\it pure} if for any $\alpha,\beta\in \Lambda$ we have
$\alpha = \beta\mod\Z^n.$ In other words, a series (in particular,
a polynomial) solution centered at the origin and with irreducible
support is called pure if it is given by the product of a monomial
and a Laurent series. A set of linearly independent
series~${\{f_{k}(x)\}}_{k=1}^r$ is called a {\it pure basis} of
the solution space of a Horn system in a neighborhood of a
nonsingular point $x\in\C^n$ if every~$f_k$ converges at~$x,$ is a
pure solution and together they span a linear space whose
dimension equals the holonomic rank of the Horn system.
\label{def:pureSolution}
\end{definition}
Since a Horn system has polynomial coefficients, it follows that
any of the Puiseux series solutions to a holonomic Horn system
can be written as a finite linear combination of pure solutions
to the same system of equations.
Here the holonomic property is necessary to ensure that the linear
combination is finite. Moreover, in a neighborhood of a
nonsingular point, a pure basis in the local solution space of a
Horn system is defined uniquely up to permutation and
multiplication of its elements with nonzero constants. In this
paper we will neglect this unessential difference between pure
bases of solutions to hypergeometric systems. If necessary, we
will explicitly specify the ordering of the elements of the pure
basis and the way they are normalized. The pure basis of a
hypergeometric system is especially convenient for computing
monodromy since, within the domain of convergence of the basis
series, the monodromy matrices are diagonal.
\begin{definition}
\rm
A Puiseux polynomial solution to the hypergeometric system Horn($A,c$) is called {\it persistent}
if its support remains finite under arbitrary small perturbations of the vector of parameters~$c.$
\label{def:PersistentPolynomialSolution}
\end{definition}
For instance, the first solution to the hypergeometric system
(\ref{horn(1,2)(-1,-1),(0,-1)}) is a persistent Puiseux monomial
since it remains monomial for any $(c_1,c_2,c_3) \in\C^3.$ The second
solution to (\ref{horn(1,2)(-1,-1),(0,-1)}) is a (Puiseux)
polynomial only for $-(c_1+c_2+c_3) \in \N$ and it is therefore not a
persistent polynomial solution. The notion is also illustrated in
Examples~\ref{ex(1,1)(1,-2),(-2,1)},
\ref{ex(1,2)(-1,-2),(-1,1),(1,-1),(-3,-2),(3,2),(2,-1),(-2,1)}
and~\ref{ex(2,-1)(2,-1),(-2,1),(-1,3),(-1,3),(1,-3),(1,2),(-1,-2),(-1,-2)}.

We will denote the linear space of all (not necessarily
persistent) Puiseux polynomial solutions to the Horn system
defined by the Ore-Sato coefficient $\varphi(s)$
by~$\Psi(\varphi)$ and use the notation~~$\Psi_{0}(\varphi)$ for
the space of all persistent polynomial solutions to this system.
The following is an immediate consequence of
Definition~\ref{def:PersistentPolynomialSolution}.
\begin{proposition}
For an Ore-Sato coefficient $\varphi$ defined by~(\ref{oresatocoeff}) with generic vector $c=(c_1,\ldots, c_m) \in \C^m$
of parameters every Puiseux polynomial solution
to the corresponding hypergeometric system {\rm Horn}($\varphi$) is persistent.
That is to say, $\Psi(\varphi)=\Psi_{0}(\varphi)$ as long as~$c$ is generic.
\end{proposition}
The next proposition is proved by analysis of the difference equations satisfied by
the coefficient of a hypergeometric polynomial (see \cite{DMS}).
\begin{proposition}
Let~$\varphi(s)$ be an Ore-Sato coefficient and let~$f(x)$ be a
Puiseux polynomial solution to {\rm Horn($\varphi$)}. If this
polynomial solution is persistent then there exists a multi-index
$I= \{i_1,\ldots, i_n \} \subset \{1,\ldots,m\}$ with different
components such that for any $s \in {\rm supp} f$ and any $\ell =
1,\ldots,n$ there exists $j \in I$ and $k \in \{0,\ldots,|A_{j,
\ell}| - 1 \}$ such that $\langle {\bf A}_j,s \rangle + c_j + k = 0$.
\label{prop:persistentPolSol}
\end{proposition}

\begin{definition}
\rm We say that the Ore-Sato coefficient
$\varphi(s) =  \prod_{i=1}^{m} \Gamma(\langle {\bf A}_{i}, s \rangle + c_{i})$
(as well as the corresponding hypergeometric system Horn($\varphi(A,c)$) is {\it resonant}
if there exists a multi-index $I=(i_1,\ldots,i_k)$ with $1\leq i_1 < \ldots < i_k \leq m,$
$1\leq k \leq m$ such that for any linear relation $a_{i_1} {\bf A}_{i_1} + \ldots + a_{i_k} {\bf A}_{i_k} = 0$
with integer and relatively prime coefficients $a_{i_1},\ldots,a_{i_k}\in\Z$
we have $a_{i_1} c_{i_1} + \ldots + a_{i_k} c_{i_k}\in\Z.$
The system Horn($\varphi(A,c)$) is called {\it maximally resonant} if the above holds for
any multi-index $I=(i_1,\ldots,i_k)$ such that the corresponding integer vectors ${\bf A}_{i_1},\ldots,{\bf A}_{i_k}$
are linearly dependent.
\label{def:ResonantHornSystem}
\end{definition}
The notion of resonance is illustrated by the following example that is based on
a hypergeometric system of the smallest possible rank.
\begin{example}
\rm
To simplify the notation, here and throughout the paper we will define a system
of linear homogeneous differential equations by giving the set of its generating operators.
The Horn system
\begin{equation}
\left\{
\begin{array}{l}
x_1 (\theta_1  + \theta_2 + c_3) - (\theta_1 + c_1), \\
x_2 (\theta_1  + \theta_2 + c_3) - (\theta_2 + c_2)
\end{array}
\right.
\label{horn(1,1),(-1,0),(0,-1)}
\end{equation}
is the only (up to a monomial change of variables defined by a
unimodular matrix) bivariate hypergeometric system whose holonomic
rank equals~1 for all values of its parameters $c_1,c_2,c_3 \in \C.$ The
only solution to this system is $x_1^{-c_1} x_2^{-c_2} (1 - x_1 -
x_2)^{c_1+c_2-c_3}.$ It is resonant (and maximally resonant as well, since
it has holonomic rank~1) if and only if $c_1+c_2-c_3 \in \Z.$ The
monodromy of~(\ref{horn(1,1),(-1,0),(0,-1)}) only depends on the
values of $a,b,c$ modulo~$\Z$ and is the subgroup of $\C$ with the
three generators $\{ \exp(2\pi \sqrt{-1}\, c_1), \exp(2\pi
\sqrt{-1}\, c_2), \exp(2\pi \sqrt{-1}\, c_3) \}$ in non-resonant case,
while it has less than two generators in resonant case (if the group is
not trivial). \label{ex(1,1)(1,0)(0,1)}
\end{example}
The crucial importance of the notion of resonance will be revealed in the theorems and examples that follow.
Roughly speaking, nonresonant parameters of a hypergeometric system mean that
any of its solutions is either a fully supported series (centered at the origin) or a persistent
Puiseux polynomial. Resonant parameters may correspond to non-holonomic systems, systems with non-persistent polynomial
solutions, non-fully supported series solutions or, possibly, logarithmic solutions which do
not admit any expansions into Puiseux series (centered at the origin) at all.
 For instance, the hypergeometric system~\eqref{horn(1,0)(0,1)(1,1)(-1,0)^2(0,-1)^2} is
maximally resonant.

\begin{definition}
\rm A solution $f(x)$ to the system of differential equations
Horn($\varphi$) at a nonsingular point $x^{(0)}\in\C^n$ is said to
generate a linear subspace $L \subset S({\rm Horn}(\varphi))|_{
V(x^{(0)})}$ of the space of all holomorphic solutions to
Horn($\varphi$) in a simply connected neighbourhood $V(x^{(0)})$
if every element of $L$ can be represented as a linear combination
of branches of $f(x)$ on $V(x^{(0)}).$ We will say that $f(x)$ is
a {\it generating solution} of $L.$ A function is called a
generating solution to a system of equations if it generates the
whole space of its holomorphic solutions at any nonsingular point.
In Section~\ref{sec:polynomialBases} we will construct generating
solutions for two families of hypergeometric systems
(Proposition~\ref{prop:simplicialSolution},
Proposition~\ref{prop:parallelepiped}).
\label{def:generatingSolution}
\end{definition}
\begin{example}
\rm
The maximally resonant Horn system defined by the Ore-Sato coefficient
$\varphi(s) = \Gamma(s_1)\Gamma(s_2) \Gamma(s_1+s_2) \Gamma(-s_1)^2 \Gamma(-s_2)^2$
is given by
\begin{equation}
\left\{
\begin{array}{l}
x_1 \, \theta_1 (\theta_1 + \theta_2) - \theta_{1}^2, \\
x_2 \, \theta_2 (\theta_1 + \theta_2) - \theta_{2}^2.
\end{array}
\right.
\label{horn(1,0)(0,1)(1,1)(-1,0)^2(0,-1)^2}
\end{equation}
This system has holonomic rank 4. Its space of holomorphic
solutions is spanned by $1,\, \log{x_1}, \, \log{x_2}, \,
\log{x_1} \log{x_2} + {\rm PolyLog}(2,x_1)$ $ + {\rm
PolyLog}(2,x_2).$ Here ${\rm PolyLog}$ $( 2, z ) =
\sum_{k=1}^{\infty} z^k/k^2.$ The resultant of the principal
symbols of (\ref{horn(1,0)(0,1)(1,1)(-1,0)^2(0,-1)^2}) equals
$x_1x_2(x_1-1)(x_2-1)(x_1+x_2-1).$ Using the properties of ${\rm
PolyLog}(2,z)$ (see~\cite{higherLogarithms}), we conclude that the
monodromy group of~(\ref{horn(1,0)(0,1)(1,1)(-1,0)^2(0,-1)^2}) is
generated by the four matrices
$$
M_{x_1=0}=
\left(
\begin{array}{cccc}
 1 & 0 & 2\pi \sqrt{-1}  & 0               \\
 0 & 1 & 0               & 2\pi \sqrt{-1}  \\
 0 & 0 & 1               & 0               \\
 0 & 0 & 0               & 1
\end{array}
\right), \quad
M_{x_2=0}=
\left(
\begin{array}{cccc}
 1 & 2\pi \sqrt{-1} & 0      & 0               \\
 0 & 1              & 0      & 0               \\
 0 & 0              & 1      & 2\pi \sqrt{-1}  \\
 0 & 0              & 0      & 1
\end{array}
\right),
$$
$$
M_{x_1=1}=
\left(
\begin{array}{cccc}
 1 & -2\pi \sqrt{-1}     & 0 & 0 \\
 0 & 1                   & 0 & 0 \\
 0 & 0                   & 1 & 0 \\
 0 & 0                   & 0 & 1
\end{array}
\right), \quad
M_{x_2=1}=
\left(
\begin{array}{cccc}
 1 & 0      & -2\pi \sqrt{-1} & 0      \\
 0 & 1      & 0               & 0      \\
 0 & 0      & 1               & 0      \\
 0 & 0      & 0               & 1
\end{array}
\right).
$$
This monodromy representation shows that $\log{x_1} \log{x_2} + {\rm PolyLog}(2,x_1) + {\rm PolyLog}(2,x_2)$
is a generating solution of $S({\rm Horn}(\varphi)).$
\label{ex(1,0)(0,1)(1,1)(-1,0)^2(0,-1)^2}
\end{example}
If the monodromy representation of the entire solution space
$S({\rm Horn} (\varphi))$ is irreducible then it admits a
generating solution. On the other hand, the monodromy
representation can be reducible for $S({\rm Horn}(\varphi))$ with
a generating function as the above
Example~\ref{ex(1,0)(0,1)(1,1)(-1,0)^2(0,-1)^2} illustrates.

The main result in the paper
(Theorem~\ref{thm:zonotopesHavePolBasis}) describes bivariate
hypergeometric systems whose solution spaces split into
one-dimensional invariant subspaces. Throughout the paper, we will
adopt the following definition.
\begin{definition} \rm
We will say that the monodromy representation of a system of equations
is {\it maximally reducible} if its solution space splits into a direct
sum of one-dimensional invariant subspaces.
\label{def:maxReducibleMonodromy}
\end{definition}


\section{The structure of the space of holomorphic solutions to a Horn system}
\label{sec:atomicHorn}

\subsection{Integral representations and calculation of multidimensional residues.}
Our main tool for computing analytic continuation of a hypergeometric series is the
Mellin-Barnes integral. The following theorem gives an integral
representation for solutions to a hypergeometric system.

\begin{theorem} {\rm (See~\cite{Sadykov-SMZh}).}
Let  $$ \psi(s)=\prod_{j=1}^m \Gamma(\langle {\bf A}_j,s \rangle + c_j)$$ be a nonconfluent Ore-Sato coefficient.
Let us put $\varphi(s) = \phi(s) \psi(s) $, where $\phi(s)$ is a periodic meromorphic function with the period~1 in every coordinate direction.
Then the Mellin-Barnes integral
\begin{equation}
MB(\varphi,{\mathcal C}):= \int\limits_{\mathcal C} \varphi(s) \, x^s ds
\label{MBSolution}
\end{equation}
represents a solution to ${\rm Horn}(A,c).$ Here~${\mathcal C}$ is any
$n$-dimensional contour which is homologous to its unitary
shifts in any real direction in the complement of the singularities
of the integrand in~(\ref{MBSolution}).
\label{thm:generalMBIntegralRepresenationForHorn}
\end{theorem}
The next proposition is proved, like the previous theorem, by
computing multidimensional residues at simple singularities. It
allows one to convert a multiple hypergeometric series into an
iterated Mellin-Barnes integral.
\begin{proposition}
Let $\psi(k)/k!$ be a nonconfluent Ore-Sato coefficient with
generic parameters, $A\in GL(n,\Z)$ an integer nondegenerate
square matrix with the rows ${\bf A}_1,\ldots,{\bf A}_n$ and $\alpha\in\C^n.$
For a sufficiently small $\varepsilon > 0$ and $k\in\N^n$ let
$\tau(k)=\{s\in\C^n : | \langle {\bf A}_j, s \rangle + \alpha_j + k_j| =
\varepsilon, {\rm\ for\ any\ } j=1,\ldots,n \}$ and define
$\mathcal{C}=\sum\limits_{k\in\midN^n} \tau(k).$ Then
$$
\sum\limits_{k\in\midN^n} \frac{{(-1)}^{|k|}}{k!} \psi(k) \,
x^{Ak+\alpha} = \frac{1}{(2\pi \sqrt{-1})^n |A|}
\int\limits_{\mathcal{C}} \prod_{j=1}^{n} \Gamma( {(-A^{-1}(s -
\alpha))}_j ) \, \psi(A^{-1}(s-\alpha)) \, x^s ds.
$$
\label{prop:seriesToMBIntegral}
\end{proposition}
The following theorem gives a solution to the hypergeometric system
${\rm Horn}(A, \alpha)$ in the form of a multiple Mellin-Barnes
integral and allows one to convert it into a hypergeometric (Puiseux)
series by computing the residues at a distinguished family of
singularities of the integrand.
\begin{theorem}
{\rm (See~\cite{Sadykov-SMZh}).} Let~$A$ be a $m\times n$ integer
matrix of full rank~$n$ with the rows ${\bf A}_1,\ldots,{\bf A}_m$ and let
$I=(i_1,\ldots,i_n)\subset\{1,\ldots,m\}$ be a multi-index such
that the matrix~${\bf A}_I$ with the rows ${\bf A}_{i_1},\ldots,{\bf A}_{i_n}$ is
nondegenerate. For a sufficiently small $\varepsilon>0$ and
$k\in\N^n$ let $\tau_I(k)=\{s\in\C^n : | \langle {\bf A}_{i_j}, s
\rangle + \alpha_{i_j} + k_j| = \varepsilon, {\it\ for\ any\ }
j=1,\ldots,n \}$ and define
$\mathcal{C}_I=\sum\limits_{k\in\midN^n} \tau_I(k).$ Then for
generic $\alpha\in\C^m$ and $\alpha_I=(\alpha_{i_1},\ldots,
\alpha_{i_n})$ the following Mellin-Barnes integral satisfies the
system of equations ${\rm Horn}(A, \alpha)$ and can be represented
in the form of a hypergeometric (Puiseux) series:
\begin{equation}
\frac{1}{(2\pi \sqrt{-1})^n} \int\limits_{\mathcal{C}_I}
\prod_{j=1}^{m} \Gamma(\langle {\bf A}_j, s \rangle + \alpha_j) \, x^s
ds
\label{MBSolutionToHorn}
\end{equation}
$$ = \sum\limits_{k\in\midN^n} \frac{{(-1)}^{|k|}}{k! |{\bf A}_I|}
\prod_{j\not\in I} \Gamma(\langle {\bf A}_j, - {\bf A}_{I}^{-1} (k+\alpha_I)
\rangle + \alpha_j) \, x^{- {\bf A}_{I}^{-1} (k+\alpha_I)}.$$
\label{thm:MBSolutionToHorn}
\end{theorem}

\subsection{Holonomic rank formulas.}

To give a proper formulation to the main
Theorem~\ref{thm:decompositionOfTheSolutionSpace} of this section,
we introduce the following notion.

\begin{definition}
\rm For $m\geq n$ let~$A$ be a $m\times n$ integer matrix of rank~$n$
with the rows ${\bf A}_1,\ldots,{\bf A}_m$ and let $c\in\C^m$ be a vector of
parameters.
Let $I=(i_1,\ldots,i_n)$ be a multi-index such that the square
matrix~${\bf A}_I$ with the rows ${\bf A}_{i_1},\ldots,{\bf A}_{i_n}$ is
nondegenerate. Let~$c_I$ denote the vector
$(c_{i_1},\ldots,c_{i_n}).$ The hypergeometric system ${\rm
Horn}({\bf A}_I,c_I)$ will be referred to as {\it an atomic system
associated with the system} ${\rm Horn}(A,c).$ The number of
atomic systems associated with a hypergeometric system ${\rm
Horn}(A,c)$ equals the number of maximal nondegenerate square
submatrices of the matrix~$A.$ \label{def:associatedAtomicSystem}
\end{definition}
It follows from Theorem~1.3 in~\cite{Sadykov-MathScand} that, as
long as the supports of series solutions are concerned, a generic
hypergeometric system is built of associated atomic systems. More
precisely, the set of supports of solutions to a hypergeometric
system with generic parameters consists of supports of solutions
to associated atomic systems. In particular, the initial exponents
of Puiseux polynomial solutions to a hypergeometric system are
precisely the initial exponents of Puiseux polynomials which
satisfy the associated atomic systems. In the following statement
we sum up the basic properties of Horn hypergeometric systems
that we will need in the sequel.

\begin{proposition}
\label{p3.5} For any solution $v(x)$ to an atomic system
associated with a nonconfluent holonomic system
${\operatorname{Horn}}(A,c)$ with a generic vector of parameters
$c\in\mathbb C^m$, there exists a solution $u(x)\in
S(\operatorname{Horn}(A,c))$ whose support coincides with the
support of the function~$v(x)$.
\end{proposition}

\begin{proof}
Consider a nonconfluent holonomic system
$\operatorname{Horn}(A,c)$ defined by the Ore-Sato coefficient
$$
\varphi(s)
=\phi(s)\prod_{i=1}^{m}\Gamma(\langle{\mathbf A}_{i},s\rangle+c_{i})
$$
with a suitable meromorphic periodic function~$\phi(s)$.

Any solution to the associated atomic system
$\operatorname{Horn}(A_{\mathbf I},c_{\mathbf I})$, ${\mathbf
I}=(i_1,\dots,i_n)\subset\{1,\dots,m\}$ admits the integral
representation
$$
v(x)=\int_{C_{\mathbf I}}\prod_{i\in{\mathbf I}}
\Gamma(\langle{\mathbf A}_{i},s\rangle+c_{i})\phi(s)x^s\,ds
$$
for a suitable choice of the contour $C_{\mathbf I}$ and the
periodic function~$\psi(s)$.

Using this integral representation we obtain the following
solution to the nonconfluent holonomic  system ${\operatorname{Horn}}(A,c)$:
$$
u(x)=\int_{C_{\mathbf I}}\prod_{i\in{\mathbf I}}
\Gamma(\langle{\mathbf A}_{i},s\rangle+c_{i})
\prod_{j\notin{\mathbf I}}
\Gamma(\langle{\mathbf A}_{j},s\rangle+c_{j})\phi(s)x^s\,ds.
$$
Since the vector of parameters $c\in\mathbb C^m$ is generic, we
may assume that the contour~$C_{\mathbf I}$ only contains
intersections of $n$ polar sets of the product
$\prod_{i\in{\mathbf I}}\Gamma(\langle{\mathbf
A}_{i},s\rangle+c_{i})$, that are moreover disjoint from the poles
of the product $\prod_{j\not\in{\mathbf I}}\Gamma (\langle{\mathbf
A}_{j},s\rangle+c_{j})\phi(s)$. Thus in a small neighborhood of
the poles of the factor $\prod_{i\in{\mathbf
I}}\Gamma(\langle{\mathbf A}_{i},s\rangle+c_{i})$ the meromorphic
function $\prod_{j \notin{\mathbf I}} \Gamma(\langle{\mathbf
A}_{j},s\rangle+c_{j})\phi(s)$ is holomorphic. This immediately
yields that the support of~$u(x)$ coincides with the support
of~$v(x)$.
\end{proof}

\begin{remark}
\label{r3.6} {\rm If the vector of parameters $c\in\mathbb C^m$ is not
generic then the support of a solution $u(x)\in
S(\operatorname{Horn}(A,c))$ to a hypergeometric system can be a
proper subset of the support to a solution $v(x)\in
S({\operatorname{Horn}}(A_I,c_I))$ of the associated atomic
system.

Consider the following example:
$$
A=((-1,2),(2,-1),(-1,-1)),
\qquad
c=(0,0,-2).
$$
Given a solution to the hypergeometric system
${\operatorname{Horn}}(A,c)$
$$
w(x)=\sum_{m,n\ge0}
{\operatorname{Res}_{\substack{-s_1+2s_2=-m \\
2{s_1}-s_2=-n}}}\Gamma(-s_1+2s_2)\Gamma(-{s_1}-{s_2}-2)
\Gamma(2{s_1}-{s_2})x^s,
$$
we define the solution to the associated atomic system
$$
v(x)=\sum_{m,n\ge0}{\operatorname{Res}_{\substack{-s_1+2s_2=-m \\
2{s_1}-s_2=-n}}}\Gamma(-s_1+2s_2)\Gamma(2{s_1}-{s_2})x^s.
$$

Since the solution space $S({\operatorname{Horn}}(A,c))$ is
invariant under the monodromy action, the function
$$
u(x)=\frac{1}{2\pi\sqrt{-1}}\bigl(w(x_1e^{2\pi\sqrt{-1}},x_2)
-w(x_1,x_2)\bigr)
$$
satisfies the system ${\operatorname{Horn}}(A,c)$. Straightforward
computation shows that
$$
u(x)=\frac{\bigl(x_1^{2/3}x_2^{2/3}+\sqrt[3]{x_1}+\sqrt[3]{x_2}\bigr)^2}
{3x_1^{4/3}x_2^{4/3}},
$$
i.e. the support of $u(x)$ consists of the six points
$\{s\in\mathbb C^2\!:\allowbreak s_1-2s_2\in\mathbb Z_{\ge0},\;
-2s_1+s_2\in\mathbb Z_{\ge0},\;-2\le s_1+s_2\le0\}$. Observe that
the meromorphic function
$\Gamma(-s_1+2s_2)\Gamma(-{s_1}-{s_2}-2)\Gamma(2{s_1}-{s_2})x^s$
has triple poles at the point that belong to the support of
$u(x)$, all the other its poles being simple.}
\end{remark}

The next theorem summarizes the main properties of the space of
holomorphic solutions to a Horn system that we need throughout the
rest of the paper.

\begin{theorem}
Assume that the hypergeometric system ${\rm Horn}(A,c)$ is nonconfluent, holonomic
and has generic vector of parameters~$c$.

{\rm (1)} The space of local holomorphic solutions at a nonsingular point~$x^{(0)}$
to ${\rm Horn}(A,c)$ admits the following decomposition:
$ S({\rm Horn}(A,c)) = \Psi\oplus \mathcal{F}_{x^{(0)}}. $
Here~$\Psi$ is the subspace of its persistent Puiseux polynomial
solutions and~$\mathcal{F}_{x^{(0)}}$ is the subspace of its fully supported
Puiseux series solutions which converge at~$x^{(0)}.$

{\rm (2)} The dimension of the space~$\mathcal{F}_{x^{(0)}}$ of Puiseux series
(centered at the origin) which satisfy ${\rm Horn}(A,c)$ and converge
at~$x^{(0)} \in {^c\mathcal{A}(\varphi(A,c))}$ is given by
$$
\dim_{\midC} \mathcal{F}_{x^{(0)}} =
\sum\limits_{
\begin{array}{l}
I=(i_1,\ldots,i_n)\subset \{1,\ldots,m\} \\
M(\varphi(A,c),\Log\, x^{(0)}) \subset {({\bf A}_{I}^{-1} \midR_{ + }^n)}^{\vee}
\end{array}
}
|\det {\bf A}_I|.
$$

{\rm (3)} The dimension of the space~$\Psi_0$ of persistent Puiseux polynomial solutions
to a bivariate system ${\rm Horn}(A,c)$ is given by
$\dim_{\midC} \Psi_0 = \sum\limits_{{\bf A}_i,{\bf A}_j \ {\rm lin.\ indep.}} \nu({\bf A}_i,{\bf A}_j). $
\label{thm:decompositionOfTheSolutionSpace}
\end{theorem}

\begin{proof}
(1) Observe that any Puiseux series solution (centered at the origin)
of a Horn system with generic parameters is either a fully supported
series or a persistent Puiseux polynomial.
Indeed, for a polynomial to be a solution to a hypergeometric system, its
exponents must satisfy a system of linear algebraic equations. The
generic parameters assumption implies that the right-hand-sides of
these equations are also generic and hence the system of linear
algebraic equations is defined by a square nondegenerate matrix.
The corresponding solutions to the hypergeometric system are
precisely persistent polynomials.
This means, in particular, that for an Ore-Sato coefficient $\varphi$ with
generic parameters $\Psi(\varphi) = \Psi_0 (\varphi).$
Since no linear combination of elements in $\Psi(\varphi)$ can yield a
fully supported Puiseux series, it follows that the sum is direct.

(2) This follows from the previous part together with the two-sided Abel lemma (see Lemma~11 in~\cite{PST})
which describes the domain of convergence of a nonconfluent hypergeometric series.
By the first part of the theorem the generic parameters assumption implies
that only fully supported series must be taken into account and it is therefore
sufficient to consider square nondegenerate submatrices of~$A$.

(3) This is the statement of Theorem~6.6 in~\cite{DMS}.
\end{proof}

The following result (see~\cite{DMS}) gives the holonomic rank of a bivariate nonconfluent
Horn system with generic parameters.
\begin{theorem} {\rm (\cite{DMS})}
Let $A$ be an~$m \times 2$ integer matrix of full rank such that
its rows ${\bf A}_1,\dots, {\bf A}_m$ satisfy ${\bf A}_1+\ldots+{\bf A}_m=0$. If $c\in
\C^m$ is a generic parameter vector, then the ideal ${\rm
Horn}(A,c)$ is holonomic. Moreover,
$$
{\rm rank}({\rm Horn}(A,c)) = \left( \sum_{i : A_{i,1} > 0} A_{i,1} \right) \cdot \left( \sum_{i : A_{i,2} > 0} A_{i,2} \right) \quad  -
\sum_{\mbox{\tiny ${\bf A}_i,{\bf A}_j$ \mbox{\rm lin. dep.} }}
\nu({\bf A}_i,{\bf A}_j),
$$
where the summation runs over linearly dependent pairs~${\bf A}_i$, ${\bf A}_j$
of rows of~$A$
that lie in opposite open quadrants of~$\Z^2.$
\label{thm:DMStheorem}
\end{theorem}

\begin{remark}
\rm The conclusion of Theorem~\ref{thm:DMStheorem} only holds
under the nonconfluency assumption on the matrix~$A.$ For
instance, the confluent Horn system generated by the operators $x_1
(\theta_1 + \theta_2) (\theta_1 + \theta_2 - a) - \theta_1$ and $x_2
(\theta_1 + \theta_2) (\theta_1 + \theta_2 - a) - \theta_2$ is
holonomic with rank~2. Indeed, if the above equations are
satisfied by a function $f(x)$ then $f_{x_1} = f_{x_2}$ and hence
$f(x)=g(x_1+x_2)$ for a suitable univariate function~$g.$ Moreover~$g(t)$
is a solution to the ordinary differential equation $t^2 g''(t) +
((1-a)t -1) g'(t) =0.$ A fundamental system of solutions of this
equation is $1, \Gamma(-a,1/t),$ where $\Gamma(p,q)$ is the
incomplete gamma-function. Thus a basis in the solution space of
the Horn system is~$1,$ $\Gamma\left(-a,\frac{1}{x_1+x_2}\right). $
Observe that $\Gamma(1,1/(x_1+x_2))=e^{-1/(x_1+x_2)}.$ Thus for a
confluent system the rank can be smaller than the product of the
degrees of the operators even if no parallel lines or persistent
polynomial solutions are present.
\end{remark}
\begin{remark}
\rm Theorem~\ref{thm:DMStheorem} is substantially bivariate, yet
it can be generalised to arbitrary dimension of the space of
variables. Theorem~6.10, 7.13 in~\cite{DMM} provide an explicit
combinatorial formula for the holonomic rank of a nonconfluent
hypergeometric system Horn$(A,c).$ Let us choose a $(m-n) \times
m$ submatrix $B$ of the matrix~$A$ with integer coefficients whose
columns span $\mathbb Z^{m-n}$ as a lattice, satisfying $B\cdot
A=0\in\mathbb Z^{m-n}\times\mathbb Z^{n}$. For
$g=|{\ker(B)}/\mathbb ZA|$ the index of the integer lattice
generated by the columns of $A$ in its saturation, the following
formula holds for generic $c\in\mathbb C^m:$
$$
\operatorname{rank}(\operatorname{Horn}(A,c))
=g\operatorname{vol}(B)+{\operatorname{rank}}(\Psi_0(\varphi)),
$$
where $\operatorname{vol}(B)$ denotes the normalised volume of the
convex hull of the columns of~$B$. This formula is a numerical
counterpart of the decomposition Theorem~\ref{thm:decompositionOfTheSolutionSpace},~1) on the
space of holomorphic solutions to a hypergeometric system.

In example~\ref{ex(1,2)(-1,-1),(0,-1)} we will see that
$\operatorname{rank}(\Psi_0)=1$, as~$\Psi_0$ is generated by~$f_1$
and the rank 
${\operatorname{rank}}$ $ (Horn (A, ( c_1, c_2,c_3)))=2$.
In fact, for $-(c_1+c_2+c_3)\notin\mathbb N$ the rank of fully
supported solutions is~1 while for $-(c_1+c_2+c_3)\in\mathbb N$
the rank of the factor space $\Psi/\Psi_0$ is~1.
\end{remark}

\subsection{Monodromy action on the invariant subspace of Puiseux polynomial solutions.}

Recall that by a Puiseux polynomial we mean a finite linear combination
of monomials with (in general) arbitrary complex exponents.
Such a polynomial may only have singularities on the union
of the coordinate hyperplanes $\{x\in\C^n : x_1 \ldots x_n =0 \}.$
The set of all Puiseux polynomial solutions of a Horn system is a linear subspace~$\Psi$
in the space of its local holomorphic solutions.
This subspace is clearly invariant under the action of monodromy.

Let~$\{ p_k(x)  \}_{k=1}^p$ be a pure basis of the linear
space~$\Psi$ (see Definition~\ref{def:pureSolution}). That is, let
$p_k(x)= x^{v_k} \tilde{p}_k(x),$ where $v_k \in\C^n$
and~$\tilde{p}_k(x)$ is a Laurent polynomial (i.e., a polynomial
with integer exponents). Since a Laurent polynomial has no
branching, it follows that the branching of this basis is the same
as that of a system of monomials $x^{v_1},\ldots, x^{v_p},$ where
$v_k\in\C^n.$ Thus the branching locus for the solutions of such a
Horn system is $\{x\in\C^n: x_1 \ldots x_n=0\},$ the generators of
the fundamental group with the base point $(1,\ldots,1)$ are
$\gamma_j=(1,\ldots,1,e^{2\pi \sqrt{-1}\, t},1,\ldots,1),$ $t\in
[0,1],$ $j=1,\ldots,n.$ The corresponding monodromy matrix is
given by $M_j= {\rm diag}(e^{2\pi \sqrt{-1}\, v_{j}}).$

\subsection{Intertwining operators for Horn systems.}
\label{subsec:intertwiningOperators}

The purpose of this subsection is to compute the intertwining
operators for the monodromy representations of Horn systems whose
parameters differ by integers. This will allow us to conclude that
certain monodromy representations are equivalent.
The intertwining operators for the monodromy representations of an
ordinary hypergeometric differential equation have been computed
in~\cite{BeukersHeckman}.

Recall that by $S({\rm Horn}(A,\alpha))$ we denote the linear space of
(local) solutions to the hypergeometric system ${\rm Horn}(A,c).$
The class of hypergeometric functions is closed under
multiplication with Puiseux monomials. More precisely, the
operator $x^\lambda \bullet$ which multiples a function with the
monomial $x^\lambda=x_{1}^{\lambda_1}\ldots x_{n}^{\lambda_n}$ is
a vector space isomorphism between the following spaces:
$$
x^\lambda \bullet : S({\rm Horn}(A, A\lambda + \alpha))
\rightarrow S({\rm Horn}(A,\alpha)).
$$
Since multiplication with a Laurent monomial does not alter the
branching of a function, we conclude that for $\lambda\in\Z^n$ the
hypergeometric systems ${\rm Horn}(A,\alpha)$ and ${\rm Horn}(A,
A\lambda + \alpha)$ have the same monodromy.
\begin{proposition}
Let ${\bf A}_1,\ldots, {\bf A}_m \in\Z^n$ be the rows of an integer matrix~$A$
of full rank~$n$ and let $c \in\C^m$ be the vector of
parameters. The differential operator
\begin{equation}
\langle {\bf A}_j, \theta \rangle + c_j - 1: S({\rm Horn}(A, c
- e_j)) \rightarrow S({\rm Horn}(A, c))
\label{intertwiningOperator}
\end{equation}
is an intertwining operator for the monodromy representations of
the corresponding Horn systems.
\label{prop:intertwiningOperators}
\end{proposition}
\begin{proof}
Denote by $H_{i}(A,c)$ the differential operator defining the
$i$-th equation in the hypergeometric system ${\rm Horn}(A,c),$ (\ref{horn}).

The following equalities immediately yield the statement: for
$A_{i,j} \leq 0$
$$
(\langle {\bf A}_j, \theta - e_i \rangle + c_j -1) H_i(A, c- e_j) =  H_i(A, c)  (\langle {\bf A}_j, \theta \rangle + c_j - 1),
$$
while for $A_{i,j} > 0$
$$
(\langle {\bf A}_j, \theta \rangle + c_j -1) H_i(A, c- e_j) =  H_i(A, c)  (\langle {\bf A}_j, \theta \rangle + c_j - 1).
$$
\end{proof}
By means of the intertwining operators, we establish a statement
analogous to Proposition~2.7 in~\cite{BeukersHeckman}.
\begin{proposition}
Suppose that the solution space of the system $S({\rm Horn}(A,c +
\ell))$ contains a nontrivial subspace of persistent Puiseux
polynomial solutions $\Psi_0 \not =\{ 0 \}$ for $\ell \in \Z^n.$
Then there is a non-trivial monodromy invariant subspace of
$S({\rm Horn}(A,c )) $ with codimension higher than~$1.$ In
particular monodromy representation of $S({\rm Horn}(A,c ))$ is
reducible.
\end{proposition}
\begin{proof}
Let $J$ be the set of indices $J \subset \{1, \ldots, m\}$ such
that $\ker (\langle {\bf A}_j, \theta \rangle + c_j + \ell_j) \cap
\Psi_0 \ni x^\alpha \not =0$ for $j \in J.$ We remark here that we
can always find a monomial element in $\Psi_0$ as long as $\Psi_0
\neq \{0\}$. Then
$$
(\langle {\bf A}_j, \theta \rangle + c_j + \ell_j) : S({\rm Horn}(A,c + \ell)) \rightarrow  S({\rm Horn}(A,c + \ell + e_j))
$$
has a non-trivial kernel. Assume $\ell_j < 0$ and choose maximal $k_j,$
$ \ell_j \leq k_j \leq -1 $ such that
$$
\langle {\bf A}_j, \theta \rangle + c_j + k_j: S({\rm Horn}(A, c + \ell + (k_j - \ell_j)e_j)) \rightarrow
S({\rm Horn}(A, c + \ell + (k_j - \ell_j +1)e_j))
$$
has a non-trivial kernel. This implies that the space
$$
\prod_{k=1}^{-k_j}(\langle {\bf A}_j, \theta \rangle + c_j - k)  S\left({\rm Horn}(A,c + \ell + (k_j - \ell_j)e_j)\right)
$$
is an invariant subspace of  $S({\rm Horn}(A,c + \ell  - \ell_j
e_j)).$

Thus  $S\left({\rm Horn}\left(A,c + \ell  - \sum\limits_{j
\in J, \ell_j <0 } \ell_j e_j \right)\right)$ has an invariant
subspace of codimension greater than~1. If we consider
$$
\prod_{i \not \in J, \ell_i <0
}\prod_{\lambda_i=0}^{-\ell_i-1}(\langle {\bf A}_i, \theta \rangle +
c_i+\ell_i + \lambda_i ) S\left({\rm Horn}\left(A,c + \ell   -
\sum_{j \in J, \ell_j <0 } \ell_j e_j \right) \right),
$$
it contains a non-trivial monodromy invariant subspace of
$$
S \left( {\rm Horn}\left( A,c + \ell   - \sum_{ \ell_j <0 } \ell_j
e_j \right) \right).
$$
Now the proof of the statement is reduced to that for the case $\ell \in \Z_{\geq 0}^n.$
We see that
$$
\prod_{j=1 }^n\prod_{\lambda_j=0}^{\ell_j-1}(\langle {\bf A}_j, \theta
\rangle + c_j+ \lambda_j )^{-1}  (S({\rm Horn}(A,c + \ell )) /
\Psi_0)
$$
is an invariant subspace of  $S({\rm Horn}(A,c ))$ in question. We remark here that
none of the operators
$\langle {\bf A}_j, \theta \rangle + c_j+ \lambda_j $ for $j = 1, \ldots, n$ and $\lambda_j=0, \ldots, \ell_j - 1$
appears in the operators  $P_i(\theta),  Q_i(\theta), \,\,\, i=1,\ldots, n $  of ~(\ref{horn}) for ${\rm Horn}(A,c + \ell)$.
\end{proof}
\begin{corollary}
In the case of two variables, suppose that
$$
\sum_{{\bf A}_j, {\bf A}_k\ {\rm lin.\,\, indep.}} \nu({\bf A}_j,{\bf A}_k)=0,
$$
where the summation is over all pairs of linearly independent rows
of the matrix defining the Horn system. Then for generic parameter
vector~$c$ the monodromy representations of the Horn systems
${\rm Horn}(A, c)$ and ${\rm Horn}(A, c - e_j)$ are
equivalent for any $j=1,\ldots,m.$
\label{cor:equivalentRepresentations}
\end{corollary}
\begin{proof}

The condition on the indices of the rows of the defining matrix
means precisely (by Theorem~\ref{thm:decompositionOfTheSolutionSpace},~3) that there are no
persistent polynomial solutions to the Horn system in question.
Thus for generic parameters all solutions are fully supported
(that is, the convex hull of the support of any of the solutions
has dimension~$2$). No such series is annihilated by a
differential operator of the form ~(\ref{intertwiningOperator})  and hence the
intertwining operators have trivial kernels. This means that the
monodromy representations are equivalent.
\end{proof}
\begin{example}
\rm The hypergeometric system defined by the matrix
\begin{equation}
\left(
\begin{array}{rr}
 1 & 2  \\
-1 & -1  \\
 0 & -1
\end{array}
\right)
\label{eq:(1,2)(-1,-1),(0,-1)matrix}
\end{equation}
and the generic parameter vector $(c_1,c_2,c_3) \in\C^3$ is generated by
the differential operators
\begin{equation}
\left\{
\begin{array}{l}
x_1 (\theta_1 + 2 \theta_2 + c_1) + (\theta_1 + \theta_2 - c_2), \\
x_2 (\theta_1 + 2 \theta_2 + c_1) (\theta_1 + 2 \theta_2 + c_1 + 1) -
(\theta_1 + \theta_2 - c_2) (\theta_2 - c_3).
\end{array}
\right.
\label{horn(1,2)(-1,-1),(0,-1)}
\end{equation}
It is holonomic for any $(c_1,c_2,c_3)$ with rank~2.
A {\it universal basis} in the solution space of~(\ref{horn(1,2)(-1,-1),(0,-1)}),
valid for any values of $(c_1,c_2,c_3)\in\C^3,$ is given by the functions
$f_1(x; c) = x_1^{c_1+2c_2} x_2^{-c_1-c_2}$
and
$f_2(x; c) = x_1^{c_1+2c_2}(x_2^{-c_1-c_2} - x_2^{c_3} (x_1+x_1^2+x_2)^{-c_1-c_2-c_3})/(c_1+c_2+c_3).$
For $c_1+c_2+c_3=0,$ this basis degenerates into the pair of functions
$ x_1^{c_1+2c_2} x_2^{-c_1-c_2},$ $x_1^{c_1+2c_2} x_2^{-c_1-c_2} \log\frac{x_1 + x_1^2 + x_2}{x_2}.$
Observe that the system~(\ref{horn(1,2)(-1,-1),(0,-1)}) is resonant if and only if $c_1+c_2+c_3\in\Z.$
The notion of maximal resonance gives nothing new in this example
since there is only one (up to scaling) linear relation between the rows of the matrix~(\ref{eq:(1,2)(-1,-1),(0,-1)matrix}).
Let $Sol(c)$ denote the linear space of
local solutions to~(\ref{horn(1,2)(-1,-1),(0,-1)}) at a
nonsingular point. The intertwining operators for this Horn system
are given by
$$
\begin{array}{lllr}
I_1 & = & \theta_1 + 2 \theta_2 + c_1-1  &: Sol(c_1-1,c_2,c_3) \rightarrow Sol(c), \\
I_2 & = & -\theta_1 - \theta_2 + c_2-1   &: Sol(c_1,c_2-1,c_3) \rightarrow Sol(c), \\
I_3 & = & - \theta_2 + c_3-1             &: Sol(c_1,c_2,c_3-1) \rightarrow Sol(c).
\end{array}
$$
Observe that
$$
\begin{array}{lcl}
I_1(f_1(x; c))     & = & I_2(f_1(x; c)) = 0, \\
I_3(f_1(x; c_1,c_2,c_3-1))     & = & (c_1+c_2+c_3-1)f_1(x;c), \\
I_1(f_2(x; c_1 - 1, c_2, c_3)) & = & I_2(f_2(x; c_1,c_2 - 1,c_3)) = \\
                           &   & (c_1 + c_2 + c_3) f_2(x; c) - f_1(x; c), \\
I_3(f_2(x; c_1, c_2, c_3 - 1)) & = & (c_1 + c_2 + c_3) f_2(x; c).
\end{array}
$$
This example shows that the intertwining operators constructed above
may have nontrivial kernels despite the fact that the monodromy
of~(\ref{horn(1,2)(-1,-1),(0,-1)}) only depends on the values of
$c_1,c_2,c_3$ modulo~$\Z.$
\label{ex(1,2)(-1,-1),(0,-1)}
\end{example}


\section{Explicit monodromy calculation for simplicial and parallelepipedal hypergeometric families}
\label{sec:polynomialBases}

\subsection{Atomic hypergeometric systems.}

In this section, we investigate monodromy representations of two families of hypergeometric systems.
They will generate two classes of polygons corresponding to maximally reducible monodromy
representations in \S ~\ref{sec:maximallyReducibleMonodromy}.


Recall that by Definition ~\ref{def:associatedAtomicSystem} an atomic hypergeometric
system of equations is a confluent Horn system defined by a
nondegenerate square matrix.
An atomic system 
can be transformed into a system of differential equations with
constant coefficients by means of the isomorphism in Corollary~5.2
in~\cite{DMS}. In accordance with the
Malgrange-Ehrenpreis-Palamodov fundamental principle
\cite{Palamodov}, an atomic system only has elementary solutions
which can be expressed in terms of Puiseux polynomials and
exponential functions. A detailed analysis of the properties of a
general atomic hypergeometric system has been carried out
in~\cite{Sadykov-Journal of SFU}. Observe that an atomic system is
confluent by definition since the nonconfluency
condition~(\ref{nonconfluency}) is a linear relation for the rows
of the defining matrix. Also, by definition an atomic system is
never resonant. A solution to a holonomic atomic system is either
a persistent Puiseux polynomial or a fully supported Puiseux
series.

In the case of two variables it is possible to tell exactly how
many Puiseux polynomial solutions an atomic system might have and
what their initial exponents are.
 The following theorem is a
consequence of  Theorem ~\ref{thm:decompositionOfTheSolutionSpace}
(2) together with Theorem~2.5, Theorem~5.3 and Lemma~6.5
in~\cite{DMS}, and the rank formula for GKZ hypergeometric system.

\begin{theorem}
{\rm (1)} For any $2\times 2$ nondegenerate integer matrix
$
M=
\left(
\begin{array}{rr}
a_1 & b_1  \\
a_2 & b_2
\end{array}
\right) $ \noindent and any $\tilde{c}\in\C^2$ the holonomic rank
of the associated {\it atomic} system is given by ${\rm rank}({\rm
Horn}(M,\tilde{c}))= |\det(M)| + \nu(M).$ Furthermore, there exist
$|\det(M)|$ fully supported series solutions of ${\rm
Horn}(M,\tilde{c})$ while the remaining~$\nu(M)$ solutions are
persistent Puiseux polynomials.

{\rm (2)} In the case when~$\nu(M)>0$, the initial exponents of
the Puiseux polynomial solutions to ${\rm Horn}(M,\tilde{c})$ are
given by~$-M^{-1}(\mathcal{R}_M + \tilde{c}),$ where
$$
\mathcal{R}_M = \left\{
\begin{array}{ll}
\{ (u,v) \in \N^2 : u < |b_1|, \, v < |a_2| \} , & \; \mbox{if} \,\;
 |a_1b_2| > |b_1a_2|, \\
\{ (u,v) \in \N^2 : u < |a_1|, \, v < |b_2| \} , & \; \mbox{if} \,\;
 |a_1b_2| < |b_1a_2|. \end{array} \right.
$$
\label{thm:bivariateAtomicHorn}
\end{theorem}

\begin{proof}
{\rm(1)}  By \cite[Proposition~4]{Sadykov-Journal of SFU} the system
${\operatorname{Horn}}(M,\tilde{c})$ admits a solution of the
following form for a suitable cycle ${\mathcal{C}}$:
\begin{equation}
\begin{split}
\label{eq4.1}
& \frac{|{\det(M)}|}{(2\pi i)^2}\int_{\mathcal{C}}
\Gamma(a_1s_1+b_1s_2+\tilde{c}_1)\Gamma(a_2s_1+b_2s_2+\tilde{c}_2)
x_{1}^{s_1}x_{2}^{s_2}\,ds_1\,ds_2
\\
& \qquad
=\sum_{k\in\mathbb Z_{\ge0}^2}\frac{{(-1)}^{|k|}}{k!}
x^{-M^{-1}(k+\tilde{c})}
=x^{-M^{-1}\tilde{c}}\sum_{k\in\mathbb Z_{\ge0}^2}\frac{1}{k!}
\prod_{j=1}^2(-x^{-M^{-1}e_j})^{k_j}
\\[-2mm]
& \qquad
=x^{-M^{-1}\tilde{c}}\exp\biggl(-\sum_{j=1}^{2}x^{-M^{-1}e_j}\biggr).
\end{split}
\end{equation}

The dimension of the linear span of the set of all analytic
continuations of~\eqref{eq4.1}, i.e. the space of the fully
supported solutions, equals $|{\det(M)}|$ since
$$
\text{G.\,C.\,D.}(\det(M),a_1,b_1,a_2,b_2)=1.
$$

By \cite[Lemma~6.5]{DMS} the dimension of the space of persistent
Puiseux polynomial solutions to the system in question is given
by~$\nu(M)$. We conclude that
${\operatorname{rank}}({\operatorname{Horn}}(M,\tilde{c}))
=|{\det(M)}|+\nu(M)$.

{\rm(2)} This statement follows from the construction of persistent
Puiseux polynomial solutions in \cite[Lemma~6.5]{DMS}.
\end{proof}

The support of a persistent polynomial solution to a bivariate
Horn system can be characterised as follows. After the above
Theorem~\ref{thm:bivariateAtomicHorn}, only submatrices ${\bf A}_I=({\bf A}_i,
{\bf A}_j)$ such that $\nu({\bf A}_i, {\bf A}_j) >0$ make contribution to persistent
solutions of ${\rm Horn}(A,\tilde{c}).$ In making the variable
change $x_1 \rightarrow \frac{1}{x_1}$ if necessary, we can assume
without loss of generality that ${\bf A}_i=(a_1,b_1) \in \N^2$ and
${\bf A}_j=(a_2,b_2) \in -\N^2.$ Furthermore if necessary we change the role
of $x_1$ and $x_2$ variables to restrict ourselves to the case
$|a_1 b_2| > |a_2 b_1 |.$ In this case $\mathcal{R}_{{\bf A}_I} = \{ (u,v) \in \N^2
: u < b_1, \, v < |a_2| \}.$

\begin{corollary}
\label{cor:persistentPolsol}
 Under the above mentioned normalisation setting, we
introduce the index set
$$
{\mathcal{\tilde R}_{{\mathbf A}_I}}=\bigl\{(u,v)\in\mathbb N^2\colon
0\le u<\min(a_1,b_1), \; 0\le v<\min(|a_2|,|b_2|)\bigr\},
$$
contained in~$\mathcal{R}_{{\mathbf A}_I}$.

{\rm(1)} The support of a persistent monomial solution of the
atomic system $\operatorname{Horn}({\mathbf A}_I,\tilde{c}_I)$ is
given by $\alpha\in-{{\mathbf A}_I}^{-1}({\mathcal{\tilde
R}_{{\mathbf A}_I}} +\tilde{c}_I)$.

{\rm(2)} We associate to each $\alpha_0\in-{{\mathbf
A}_I}^{-1}((\mathcal{R}_{{\mathbf A}_I} \setminus{\mathcal{\tilde
R}_{{\mathbf A}_I}})+\tilde{c}_I)$ a series of indices
$S_{\alpha_0}:=\bigcup_{k=0}^{K}\{\alpha_k\}$ that will be defined
later in the proof.

The support of a persistent polynomial solution to
$\operatorname{Horn}({\mathbf A}_I,\tilde{c}_I)$ is the union
of~$S_{\alpha_0}$ and the supports of persistent monomial
solutions.
\end{corollary}

\begin{proof}

We first remark that under the above mentioned normalisation, the
condition $\alpha\in-{{\mathbf A}_I}^{-1}(\mathcal{R}_{{\mathbf
A}_I}+\tilde{c})$ means that $P_2(\alpha)=0$ and $Q_1(\alpha)=0$.
The cardinality of the set of the lattice points satisfying this
condition is equal to $|a_2b_1|$.

{\rm(1)} If $\alpha\in-{{\mathbf A}_I}^{-1}({\mathcal{\tilde
R}_{{\mathbf A}_I}} +\tilde{c}_I)$, then
$$
\alpha\in\ker\bigl(\langle{\mathbf A}_i,\theta\rangle+\tilde{c}_i+u_i\bigr)
\cap\ker\bigl(\langle{\mathbf A}_j,\theta\rangle+\tilde{c}_j+v_j\bigr)
$$ for
$(u_i,v_j)\in{\mathcal{\tilde R}_{{\mathbf A}_I}}$, and hence the
operator $\langle{\mathbf A}_i,\theta\rangle+\tilde{c}_i+u_i$,
$u_i<\min(a_1,b_1)$ is a factor in
both~$P_1(\theta)$and~$P_2(\theta)$. In a similar way
$\langle{\mathbf A}_j,\theta\rangle+\tilde{c}_j+v_j$,
$v_j<\min(|a_2|,|b_2|)$, is a factor in both~$Q_1(\theta)$
and~$Q_2(\theta)$.

Let us set $i=1$ if $(\mathcal{R}_{{\mathbf
A}_I}\setminus{\mathcal{\tilde R}_{{\mathbf A}_I}}) \cap\mathbb
N\times\{0\}\ne\varnothing$, and define as usual $e_1=(1,0)$. We
similarly set $i=2$ if $(\mathcal{R}_{{\mathbf A}_I}
\setminus{\mathcal{\tilde R}_{{\mathbf A}_I}})\cap\{0\}\times
\mathbb N\ne\varnothing$, and $e_2=(0,1)$.

{\rm(2)} If $|b_2|<|a_2|$, the case $i=2$ arrives. Therefore there
exists~$\alpha_0$ such that $P_2(\alpha_0)=Q_1(\alpha_0)=0$, but
$Q_2(\alpha_0)\ne0$. The following equalities hold:
\begin{gather*}
H_2({\mathbf A}_I,\tilde{c}_I)x^{\alpha_0}
=(x_2P_2(\theta)-Q_2(\theta))x^{\alpha_0}
=-Q_2(\alpha_0)x^{\alpha_0},
\\
H_2({\mathbf A}_I,\tilde{c}_I)x^{\alpha_0-e_2}
=P_2(\alpha_0-e_2)x^{\alpha_0}-Q_2(\alpha_0-e_2)x^{\alpha_0-e_2}.
\end{gather*}

Let us now consider the sequence of integer points
$\alpha_0,\alpha_1=\alpha_0-e_2,\dots$ such that
$\alpha_k-\alpha_{k+1}=-e_1$ or~$e_2$. The points~$\alpha_k$ lie
inside the cone $C(i,j):=\{s\colon\langle{\mathbf
A}_j,s\rangle+\tilde{c}_j\le0\} \cap\{s\colon\langle{\mathbf
A}_i,s\rangle+\tilde{c}_i\le0\}$. This sequence must terminate at
a certain step and hence the union of all points
$\{\alpha_k\}_{k\ge0}$ defines a finite subset of~$C(i,j)$. Thus
for a finite set of integer points~$S_{\alpha_0}$ the linear
combination of polynomials $H_2({\mathbf
A}_I,\tilde{c}_I)x^{\alpha_k}$ (respectively $H_1({\mathbf
A}_I,\tilde{c}_I)x^{\alpha_k}$), $k=1,\dots,K$, is identically
equal to zero (see.~\cite[Lemma~6.5, Fig.~2]{DMS} depicting the
process that is equivalent to the construction of~$S_{\alpha_0}$).
If $|a_2|\le|b_2|$ and $a_1\ge b_1$ then ${\mathcal{\tilde
R}_{{\mathbf A}_I}}={\mathcal{ R}_{{\mathbf A}_I}}$. Thus all
persistent polynomial solutions are actually monomials.

If $|a_2|\le|b_2|$ and $a_1<b_1$, then the case $i=1$ arrives.
Similarly to the case $i=2$, we obtain a polynomial solution
supported in the set of integer points
$S_{\alpha_0}=\bigcup_{k\ge0}\{\alpha_k\}$,
$\alpha_1=\alpha_0-e_1,\dots$, such that
$\alpha_k-\alpha_{k+1}=-e_2$ or $\alpha_k-\alpha_{k+1}=e_1$.
\end{proof}

\begin{example}
\rm  The atomic hypergeometric system defined by the matrix
$$
M=\left(
\begin{array}{rr}
3 & 2  \\
 -4 & -3
\end{array}
\right)
$$
and the zero parameter vector has the form
\begin{equation}
\left\{
\begin{array}{l}
x_1 (3\theta_1 + 2\theta_2)(3\theta_1 + 2\theta_2+1)(3\theta_1 + 2\theta_2+2) - \\
 \phantom{--------} (-4\theta_1 -3 \theta_2)(-4\theta_1 -3 \theta_2+1)(-4\theta_1 -3 \theta_2+2) (-4\theta_1 -3 \theta_2+3), \\
x_2 (3\theta_1 + 2\theta_2)(3\theta_1 + 2\theta_2+1) - (-4\theta_1 -3 \theta_2)(-4\theta_1 -3 \theta_2+1)(-4\theta_1 -3 \theta_2+2).
\end{array}
\right.
\label{horn(3,2)(-4,-3)}
\end{equation}
After Theorem ~\ref{thm:bivariateAtomicHorn} (1), the dimension of
persistent solutions space is 8.

The persistent monomial solutions are given by
$$1,\, x_1^{-2}x_2^3, \, x_1^{-4}x_2^6, \, x_1^{-3}x_2^4, \, x_1^{-5}x_2^7, \, x_1^{-7}x_2^{10}. $$

The polynomials
$$
x_1^{-6}x_2^8-\frac{1}{3}x_1^{-6}x_2^9,\quad
x_1^{-9}x_2^{13}-4x_1^{-9}x_2^{12}+x_1^{-8}x_2^{13}+12x_1^{-8}x_2^{11}
$$
are the essentially polynomial persistent solutions.
We remark here that $(-6,9) \in  -{M}^{-1}\left(\mathcal{R}_{M} \setminus {\mathcal{\tilde R}_{M}}\right)$
and the solution is binomials in view of $|b_2|-|a_2|=1.$
\end{example}

Observe that any Puiseux polynomial solution to an atomic system
is necessarily persistent. This is of course not the case for an
arbitrary hypergeometric system.

\subsection{Simplicial hypergeometric configurations.}
\label{subsec:simplicialHorn}

An important special instance of a general nonconfluent Horn
system is the system defined by a matrix whose rows are the
vertices of an $n$-dimensional integer simplex. More precisely,
let $M\in GL(n,\Z)$ be an integer nondegenerate square matrix and
$\alpha\in\C^{n}$ a parameter vector. Let
$\tilde{\alpha}=(\alpha,\alpha_{n+1})\in\C^{n+1}.$ Denote by
$M_1,\ldots,M_n$ the rows of the matrix~$M$ and let
$M_{n+1}=-M_1-\ldots- M_n.$ Let~$\tilde{M}$ be the $(n+1)\times n$
matrix with the rows $M_1,\ldots,M_{n+1}.$ The (nonconfluent) Horn
system ${\rm Horn}(\tilde{M},\tilde{\alpha})$ associated with this
data will be called {\it simplicial.}
\begin{proposition} (See \cite{Sadykov-Doklady2008}.)
Let us assume that the parameter vector $\tilde{\alpha}$ is in
generic position. A holonomic simplicial hypergeometric system
${\rm Horn}(\tilde{M},\tilde{\alpha})$ admits the following
solution:
\begin{equation}
x^{-M^{-1}\alpha} \left(1 + \sum_{j=1}^{n} x^{-M^{-1} e_j} \right)^{- |\tilde{\alpha}|},
\label{simplicialSolution}
\end{equation}
where $e_j=(0,\ldots,1,\ldots,0)$ ($1$ in the~$j$-th position).
Any solution to the Horn system ${\rm
Horn}(\tilde{M},\tilde{\alpha})$ is either in the linear span of
analytic continuations of~(\ref{simplicialSolution}) or is a
persistent Puiseux polynomial. For $- |\tilde{\alpha}| \in
\N\setminus\{0\}$ the monodromy representation of ${\rm
Horn}(\tilde{M},\tilde{\alpha})$ is maximally reducible.
\label{prop:simplicialSolution}
\end{proposition}
\begin{example}
\label{e4.5}
\rm The Horn system
\begin{equation}
\begin{gathered}
\label{eq:horn(1,1)(1,-2),(-2,1)}
\left\{
\begin{array}{l}
x_1 (\theta_1  + \theta_2 - 3) (\theta_1 - 2 \theta_2 - 1)  -
(-2\theta_1  + \theta_2) (-2\theta_1  + \theta_2 -1), \\
x_2 (\theta_1  + \theta_2 - 3) (-2\theta_1  + \theta_2 - 1) -
(\theta_1  - 2\theta_2) (\theta_1 -2 \theta_2 - 1)
\end{array}
\right.
\end{gathered}
\end{equation}
is holonomic with rank~4. The pure basis in its
solution space is given by the Puiseux polynomials $ 1/(x_1 x_2),\ 4 +
2x_1 + 2y + 6x_1 x_2 + x_1^2 x_2 + x_1 x_2^2, $
$$
x_1^{-2/3}x_2^{-1/3}
(5 + 10 x_1 + 30 x_1 x_2 + 20 x_1^2 x_2 + x_1^3 x_2 + 5 x_1 x_2^2 + 10 x_1^2 x_2^2),
$$
$$
x_1^{-1/3}x_2^{-2/3}
(5 + 10 x_2 + 30 x_1 x_2 + 20x_1 x_2^2 + x_1 x_2^3 + 5x_1^2 x_2 + 10x^2 x_2^2).
$$
If we consider the Mellin-Barnes integral  for the following Ore-Sato coefficient
with generic $c \in \R$ along a proper integration contour $\mathcal C$,
$$
\varphi(s) =\frac{\Gamma (-c+s_1-2 s_2-1) \Gamma (-2 s_1+s_2-1)
e^{\sqrt{-1}\, \pi (s_1+ s_2)}}{\Gamma (-s_1-s_2+4)},
$$
we get a residue that represents a fully supported solution to a
Horn system obtained as a perturbation
of~(\ref{eq:horn(1,1)(1,-2),(-2,1)}) i.e. the result of
replacement of $\theta_1 -2 \theta_2$ by $\theta_1 -2 \theta_2-c:$
$$
f_c= x_1^{-\frac{c}{3}-1} x_2^{-\frac{2 c}{3}-1} \left( x_1^{2/3}x_2^{1/3}+x_1^{1/3}x_2^{2/3} +1\right)^{5-c}.
$$
Observe that for $c=0$ we get the Puiseux polynomial solution,
$$f_0=\frac{\left( x_1^{2/3} x_2^{1/3}+x_1^{1/3}x_2^{2/3}+1\right)^5}{ x_1 x_2}.$$
The reason for this phenomenon lies in the fact that in
$\varphi(s)$  the poles of the numerator $\Gamma (-c+s_1-2 s_2-1)$
are not cancelled by those of the denominator $\Gamma (-s_1-s_2+4)$
for generic~$c.$ For $c=0$ the half-space cancellation of poles
(see Definition~\ref{halfspace}) happens and the poles are located
in the strip $\{s : -2 \leq s_1+ s_2 \leq 3 \}$.

Linear combinations of several analytic continuations of $f_0$
produce three Puiseux polynomial solutions
to~(\ref{eq:horn(1,1)(1,-2),(-2,1)}) except the first one. The
only persistent solution in this example is the Laurent monomial $
1/(x_1  x_2)$ $\in$ $\ker\; (\theta_1 -2 \theta_2-1) \cap \ker\;
(-2\theta_1 + \theta_2-1)$. It means that this solution generates
a one-dimensional irreducible subspace of $\Psi_0$ with respect to
the monodromy action.
\begin{figure}[htbp]
\vbox{
\begin{minipage}{5.7cm}
\begin{picture}(160,160)
  \put(40,0){\vector(0,1){160}}
  \put(0,40){\vector(1,0){160}}
  \put(150,28){\small $s$}
  \put(28,150){\small $t$}
  \put(140,0){\line(-1,1){140}}
  \put(10,0){\line(1,2){70}}
  \put(20,0){\line(1,2){70}}
  \put(0,10){\line(2,1){140}}
  \put(0,20){\line(2,1){140}}
  \put(20,20){\circle{7}}
  \put(40,40){\circle*{3}}
  \put(40,60){\circle*{3}}
  \put(60,40){\circle*{3}}
  \put(60,60){\circle*{3}}
  \put(60,80){\circle*{3}}
  \put(80,60){\circle*{3}}
  \put(27,33){\circle*{5}}
  \put(47,33){\circle*{5}}
  \put(47,53){\circle*{5}}
  \put(47,73){\circle*{5}}
  \put(67,53){\circle*{5}}
  \put(67,73){\circle*{5}}
  \put(87,53){\circle*{5}}
  \put(33,27){\circle{5}}
  \put(33,47){\circle{5}}
  \put(53,47){\circle{5}}
  \put(73,47){\circle{5}}
  \put(53,67){\circle{5}}
  \put(73,67){\circle{5}}
  \put(53,87){\circle{5}}
\end{picture}
\vskip0.1cm \centerline{\tiny a): The supports of solutions to (\ref{eq:horn(1,1)(1,-2),(-2,1)})}
\end{minipage}
\hskip1cm
\begin{minipage}{5cm}
\begin{picture}(80,135)
  \put(20,0){\vector(0,1){80}}
  \put(0,20){\vector(1,0){80}}
  \put(20,20){\line(1,2){20}}
  \put(20,20){\line(2,1){40}}
  \put(40,60){\line(1,-1){20}}
  \put(40,40){\circle*{3}}
  \put(40,60){\circle*{3}}
  \put(60,40){\circle*{3}}
  \put(20,20){\circle*{3}}
\end{picture}
\vskip0.1cm {\tiny b): The polygon of the Ore-Sato\\ coefficient defining (\ref{eq:horn(1,1)(1,-2),(-2,1)})}
\end{minipage}
} 
\caption{}
\label{fig:supportOfSolsToEx(1,1)(1,-2),(-2,1)}
\end{figure}
\vskip0.5cm \label{ex(1,1)(1,-2),(-2,1)}
\end{example}
\begin{example}
\rm
Let us consider the bivariate ($n=2$) simplicial hypergeometric system generated by the matrix
$$
M =
\left(
\begin{array}{rr}
-2 &  0 \\
 0 & -2
\end{array}
\right)
$$
and the vector of parameters $\tilde{\alpha}=(0,0,c)$ in the sense
of the definition in the beginning of this subsection. This choice
of the parameters does not affect the generality of the present
example since changing the first two coordinates
of~$\tilde{\alpha}$ only results in a shift of the exponent space.
This system is generated by the differential operators
\begin{equation}
\left\{
\begin{array}{l}
x_1(2\theta_1 + 2\theta_2 + c)(2\theta_1 + 2\theta_2 + c + 1) - 2\theta_1 (2\theta_1 - 1), \\
x_2(2\theta_1 + 2\theta_2 + c)(2\theta_1 + 2\theta_2 + c + 1) - 2\theta_2 (2\theta_2 - 1).
\end{array}
\right.
\label{eq:horn(-2,0)(0,-2)(2,2)}
\end{equation}
By Theorem~\ref{thm:DMStheorem} the holonomic rank of~(\ref{eq:horn(-2,0)(0,-2)(2,2)}) equals~4.
By Proposition~\ref{prop:simplicialSolution} the generating solution to~(\ref{eq:horn(-2,0)(0,-2)(2,2)})
is given by $(1 + \sqrt{x_1} + \sqrt{x_2})^{-c}.$ It follows from Theorem~\ref{thm:decompositionOfTheSolutionSpace}
that~(\ref{eq:horn(-2,0)(0,-2)(2,2)}) does not admit any persistent Puiseux polynomial solutions
and therefore for generic $c\in\C$ a basis in the space of analytic solutions to~(\ref{eq:horn(-2,0)(0,-2)(2,2)})
is given by
\begin{equation}
\left\{
\begin{array}{lcl}
f_1(c) & = & (1 + \sqrt{x_1} + \sqrt{x_2})^{-c}, \\
f_2(c) & = & (1 + \sqrt{x_1} - \sqrt{x_2})^{-c}, \\
f_3(c) & = & (1 - \sqrt{x_1} + \sqrt{x_2})^{-c}, \\
f_4(c) & = & (1 - \sqrt{x_1} - \sqrt{x_2})^{-c}.
\end{array}
\right.
\label{eq:genericSolsOf(-2,0)(0,-2)(2,2)}
\end{equation}
However, this basis degenerates for two special values of~$c,$
namely for $c=0$ (when all the basis
elements~(\ref{eq:genericSolsOf(-2,0)(0,-2)(2,2)}) are identically
equal to~$1$) and for $c=-1$ (when $f_1(-1) - f_2(-1) - f_3(-1) +
f_4(-1) \equiv 0$). Let us furnish bases in the solution space
of~(\ref{eq:horn(-2,0)(0,-2)(2,2)}) for both of these resonant
values of the parameter~$c.$

If $c=-1,$ the corresponding resonant basis is given by $f_1(-1), f_2(-1), f_3(-1)$ and
the function
$$
\tilde{\tilde{f}}_4 = \left( f_1 \log f_1  - f_2 \log f_2 - f_3
\log f_3 + f_4 \log f_4 \right)  \big|_{c=-1}.
$$
For $c=0,$ a basis in the solution space
of~(\ref{eq:horn(-2,0)(0,-2)(2,2)}) is given by $f_1(0)$ and the
three additional resonant solutions
$$
\begin{array}{l}
\tilde{f}_2 = \log (1 + \sqrt{x_1} + \sqrt{x_2}) - \log(1 + \sqrt{x_1} - \sqrt{x_2}), \\
\tilde{f}_3 = \log (1 + \sqrt{x_1} + \sqrt{x_2}) - \log(1 - \sqrt{x_1} + \sqrt{x_2}), \\
\tilde{f}_4 = \log (1 + \sqrt{x_1} + \sqrt{x_2}) - \log(1 - \sqrt{x_1} - \sqrt{x_2}).
\end{array}
$$
However, it turns out to be possible to construct a single
universal basis in the space of analytic solutions
to~(\ref{eq:horn(-2,0)(0,-2)(2,2)}) whose elements remain linearly
independent after passing to the limit as $c\rightarrow 0$ or
$c\rightarrow -1.$ This basis has the following form:
\begin{equation}
{
\begin{array}{lcl}
\hat{f}_1(c) & = & \left(1 + \sqrt{x_1} + \sqrt{x_2}\right)^{-c}, \\
\hat{f}_2(c) & = & \left((1 + \sqrt{x_1} + \sqrt{x_2})^{-c} - (1 + \sqrt{x_1} - \sqrt{x_2})^{-c}\right)/c, \\
\hat{f}_3(c) & = & \left((1 + \sqrt{x_1} + \sqrt{x_2})^{-c} - (1 - \sqrt{x_1} + \sqrt{x_2})^{-c}\right)/c, \\
\hat{f}_4(c) & = & \left((1 + \sqrt{x_1} + \sqrt{x_2})^{-c} - (1 + \sqrt{x_1} - \sqrt{x_2})^{-c} - \right. \\
             &   &  \phantom{---------.} \left. (1 - \sqrt{x_1} + \sqrt{x_2})^{-c} + (1 - \sqrt{x_1} - \sqrt{x_2})^{-c}\right)/(c + c^2).
\end{array}
}
\label{eq:univBasisFor(-2,0)(0,-2)(2,2)}
\end{equation}
It is easy to check that the functions $\hat{f}_1(c),\ldots,\hat{f}_4(c)$
are linearly independent for any $c\in\C.$

Given the basis~(\ref{eq:univBasisFor(-2,0)(0,-2)(2,2)}), it is straightforward to find the monodromy
representation of the fundamental group of the complement to the singularities of the solutions to~(\ref{eq:horn(-2,0)(0,-2)(2,2)}).
It is generated by three matrices corresponding to the loops around the coordinate axes $\{x_1=0\},$ $\{x_2=0\}$ and the
essential singularity $\{ \mathcal{S}(x):= 1 - 2 x_1 + x_1^2 - 2 x_2 - 2 x_1 x_2 + x_2^2 = 0 \}.$
These matrices are given by
$$
M_{x_1} =
\left(
\begin{array}{ccrc}
1 & 0 & -c &  0      \\
0 & 1 &  0 & -1 - c  \\
0 & 0 & -1 &  0      \\
0 & 0 &  0 & -1
\end{array}
\right),
\quad
M_{x_2} =
\left(
\begin{array}{crrc}
1 & -c &  0 &  0      \\
0 & -1 &  0 &  0      \\
0 &  0 &  1 &  -1-c   \\
0 &  0 &  0 & -1
\end{array}
\right),$$
$$ \quad M_{\mathcal{S}} = {\rm diag} (e^{-2\pi \sqrt{-1}\,
c}).
$$
\label{ex:(-2,0)(0,-2)(2,2)simplicial}
\end{example}

\subsection{Parallelepipedal hypergeometric configurations.}

Let $M\in GL(n,\Z)$ be an integer nondegenerate square matrix and
let $\alpha,\beta\in\C^{n}$ be two parameter vectors. Denote
by~$\tilde{M}$ the $2n\times n$ matrix obtained by joining
together the rows of the matrices~$M$ and~$-M.$ The rows of such a
matrix define the vertices of a parallelepiped of nonzero
$n$-dimensional volume. Let~$\tilde{\alpha}$ be the vector with
the components
$(\alpha_1,\ldots,\alpha_n,\beta_1,\ldots,\beta_n).$
It turns out that the corresponding Horn system ${\rm Horn}(\tilde{M},\tilde{\alpha})$
admits a simple basis of solutions.
\begin{proposition} (See \cite{Sadykov-Journal of SFU}.)
Let us assume that the parameter vector $\tilde{\alpha}$ is in
generic position. The holonomic hypergeometric system ${\rm
Horn}(\tilde{M},\tilde{\alpha})$ admits the following solution:
\begin{equation}
x^{-M^{-1}\alpha} \prod_{j=1}^n \left(1 + x^{-M^{-1} e_j}
\right)^{-\alpha_j - \beta_j},
\label{parallelepipedSolution}
\end{equation}
where $e_j=(0,\ldots,1,\ldots,0)$ ($1$ in the~$j$-th position).
Any solution to the hypergeometric system ${\rm Horn}(\tilde{M},\tilde{\alpha})$
is either in the linear span of analytic continuations
of~(\ref{parallelepipedSolution}) or is a persistent Puiseux
polynomial. If $-\alpha_j - \beta_j \in \N\setminus\{0\}$ for any
$j=1,\ldots,n$ then the monodromy representation of ${\rm
Horn}(\tilde{M},\tilde{\alpha})$ is maximally reducible.
\label{prop:parallelepiped}
\end{proposition}


\section{Bases in the solution space of the Horn system
\label{sec:bases}}

Let us denote by~$q$ the number of vertices of the Newton
polytope of the polynomial which defines the singular hypersurface
of the hypergeometric system under study.
In this section we construct a family of~$q$ bases in the space of
fully supported solutions to that hypergeometric system.
This result will be used in Section~\ref{sec:maximallyReducibleMonodromy}
to deduce the main result of the paper.
\begin{definition}
\label{def:amoeba}
\rm
The {\it amoeba}~$\mathcal{A}_f$ of a Laurent polynomial~$f(x)$
(or of the algebraic hypersurface $f(x)=0$) is defined to be the
image of the hypersurface~$f^{-1}(0)$ under the map ${\rm Log } :
(x_1,\ldots,x_n)\mapsto (\log |x_1|,\ldots,\log |x_n|).$
\end{definition}
Let~$\mathcal{A}(\varphi)$ denote the amoeba of the singularity of
the hypergeometric system ${\rm Horn}(\varphi).$

\begin{definition}\label{def:recessionCone} \rm
For a convex set $B\subset\R^n$ its {\it recession cone}~$C_{B}$ is
defined to be $C_{B} = \{ s\in\R^n : u + \lambda s \in B, \, \forall u\in B,
\lambda\geq 0\}$. That is, the recession
cone of a convex set is the maximal element (with respect to inclusion)
in the family of those cones whose shifts are contained in this set.
\end{definition}

The following theorem (cf. the results in~\cite{gkz89} for the
Gelfand-Kapranov-Zelevinsky system) shows that for any vertex of
the Newton polygon of the singularity of a bivariate
hypergeometric function there exists a basis in the solution space
of the corresponding Horn system. This basis consists of
hypergeometric series which converge on the preimage of the amoeba
complement which corresponds to that vertex.

\begin{theorem}
{\rm (1)} For any bivariate nonconfluent Ore-Sato coefficient~$\varphi$
with generic parameters and any connected component~$M$ of
$^c\!\mathcal{A}(\varphi)$ there exists a pure Puiseux series
basis $f_{M,i},$ $i=1,\ldots, {\rm rank}({\rm Horn}(\varphi))$ in the
solution space of ${\rm Horn}(\varphi)$ such that the recession
cone of the support of $f_{M,i}$ is contained in~$-C^{\vee}_{M}.$

{\rm (2)} The domain of convergence of the series~$f_{M,i}$ contains
${\rm Log}^{-1}(M)$ for any $i=1,\ldots, {\rm rank}({\rm Horn}(\varphi)).$
\label{thm:eachComponentHasABasis}
\end{theorem}
\begin{proof}
Let the Ore-Sato coefficient defining the Horn system be of the
form
$$
\varphi(s) = \prod_{i=1}^m \Gamma(a_i s_1 + b_i s_2 + c_{i}),
$$
where $(a_i,b_i)\in\Z^2,$ $\sum_{i=1}^m (a_i,b_i)=(0,0)$ and
$c=(c_1,\ldots,c_m)\in\C^m$ is a generic parameter vector.
By Theorem~2 in~\cite{Sadykov-Bulletin} the vectors $\{ (a_i,b_i) \}_{i=1}^{m}$ are
the normals to all sides of the polygon~$\mathcal{P}({\varphi})$
of the Ore-Sato coefficient~$\varphi$ (observe that some of them may coincide).
This theorem also implies that the number of different vectors in this set equals~$q.$
To simplify the notation, we denote the different elements in this set of
outer normals to $\mathcal{P}({\varphi})$ by $(a_1,b_1),\ldots,(a_q,b_q).$
We may without loss of generality assume that these normals are ordered counterclockwise
from $(a_1,b_1)$ to $(a_m,b_m).$
Let~$v_i$ denote the vertex of~$\mathcal{P}({\varphi})$ that joins the sides with the normals
$(a_i,b_i)$ and $(a_{i+1},b_{i+1})$ ($v_{m}$ being the vertex that
joins the first and the last sides of the polygon).
\begin{figure}[ht!]
\centering
\includegraphics[width=7cm]{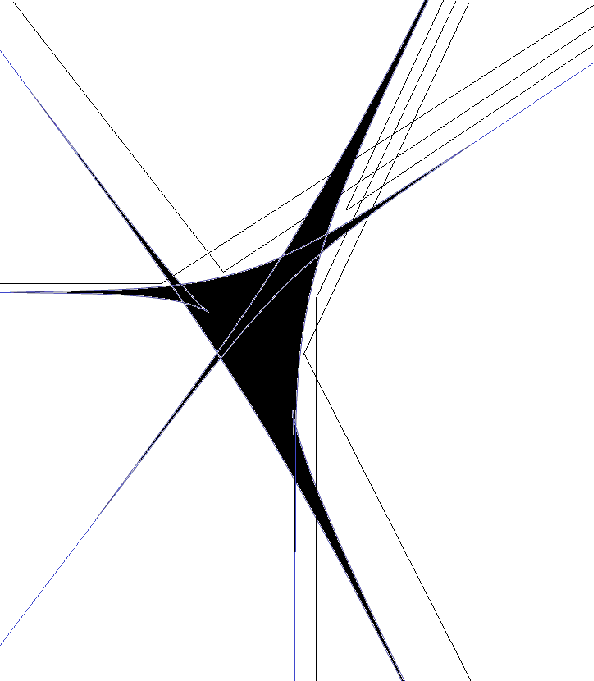}
\vskip-4cm
\hskip6cm \text{$M_1$}
\vskip4cm
\vskip-8cm
\hskip4cm \text{$M_2$}
\vskip8cm
\vskip-9cm
\hskip-2cm \text{$M_3$}
\vskip9cm
\vskip-7.5cm
\hskip-6cm \text{$M_4$}
\vskip7.5cm
\vskip-5cm
\hskip-6cm \text{$M_5$}
\vskip5cm
\vskip-3cm
\hskip-1.4cm \text{$M_6 \phantom{-----} M_7$}
\caption{\small The amoeba of the singularity of a Horn system}
\label{fig:amoebaPicture}
\end{figure}
By Theorem~7 in~\cite{PST} there is a one-to-one correspondence
between the vertices $v_1,\ldots,v_q$ and the connected components
of the complement of~$\mathcal{A}(\varphi).$
Let $M_1,\ldots, M_q$ be the connected components of the complement of~$\mathcal{A}(\varphi).$

In Figure~\ref{fig:amoebaPicture} we depict the special case of
the amoeba of the singularity of the Horn system defined by the
Ore-Sato coefficient $\Gamma(s_1 + 2 s_2) \Gamma(s_1 - 2 s_2) \Gamma(-s_1 +
3 s_2) \Gamma(-s_1 - 3 s_2) \Gamma(s_1) \Gamma(-s_1 - s_2) \Gamma(s_2).$ In this
case $q=7.$ The continuous curve that bounds the amoeba and goes
inside is its contour (see \cite{PT}). The shape of the amoeba was
found by means of the Horn-Kapranov parametrisation (\cite{Tan07})
using computer algebra system Mathematica~9.0.
Figure~\ref{fig:amoebaPicture} also shows the recession cones of
the convex hulls of the connected components of the amoeba
complement that are strongly convex and contain $M_2.$ The duals
of these cones support hypergeometric series whose domains of
convergence contain ${\rm Log}^{-1} M_2.$ To prove the theorem, we
need to show that the number of such series is independent of the
connected component of the amoeba complement.

Let us prove that for any $i=1,\ldots,q$ the number of
fully supported Puiseux series solutions to ${\rm Horn}(\varphi)$
which converge on ${\rm Log}^{-1} (M_i)$ is the same. To prove
this, we will show that the number of such series whose domain of
convergence is ${\rm Log}^{-1} (M_1)$ coincides with the number of
Puiseux series solutions that converge on ${\rm Log}^{-1} (M_2).$
Repeating this argument, one can prove that for any two adjacent
components in the complement of~$\mathcal{A}(\varphi)$ the number
of Puiseux series solutions that converge on preimages of these
components under the map ${\rm Log}$ is the same. This will prove
that any such connected component carries the same number of fully
supported Puiseux series solutions.

Let us define the single-valued branch~$\arg$ of the argument
function~${\rm Arg}$ by setting $\arg(-a_2 - b_2 \sqrt{-1})=0,$
and $\lim\limits_{\varepsilon\rightarrow 0^{-}} \arg
e^{\sqrt{-1}\,\varepsilon}(-a_2 - b_2 \sqrt{-1}) = 2\pi.$ We
introduce the partial order~$\prec$ on~$\Z^2$ by saying that
$(a,b)\prec (c,d)$ if $\arg (a + b \sqrt{-1}) < \arg (c +
d\sqrt{-1}).$ We will say that $(a,b)\preccurlyeq (c,d)$ if $\arg
(a + b \sqrt{-1}) \leq \arg (c + d \sqrt{-1}).$

By Lemma~11 in~\cite{PST} and
Theorem~\ref{thm:bivariateAtomicHorn} the number of fully
supported Puiseux series solutions to the hypergeometric system
${\rm Horn}(\varphi)$ that converge in the domain ${\rm Log}^{-1}
(M_i)$ equals
$$
S_i=\sum_{\substack{j\colon-({\bar a}_{i+1},{\bar b}_{i+1})
\prec({\bar a}_j,{\bar b}_j)\preccurlyeq({\bar a}_i,{\bar b}_i), \\
\ell\colon({\bar a}_{i+1},{\bar b}_{i+1})\preccurlyeq
({\bar a}_\ell,{\bar b}_\ell)\prec-({\bar a}_j,{\bar b}_j)}}
k_jk_\ell\begin{vmatrix}
{\bar a}_\ell & {\bar b}_\ell
\\
{\bar a}_j & {\bar b}_j
\end{vmatrix},
$$
where~$k_{j}$ is the number of elements in the set of vectors
$\{(a_1,b_1),\dots,(a_{m},b_m)\}$,that coincide with~$({\bar
a}_j,{\bar b}_j)$.
Observe that by our choice of the indices of summation all of the
involved determinants are positive. To prove that $S_1=S_2$ we
make use of the fact that these two sums have many common terms.
Indeed, the sum of terms in~$S_1$ that are not present in~$S_2$ is
given by
\begin{equation}
\label{det1}
\sum_{j\colon-({\bar a}_2,{\bar b}_2)\prec({\bar a}_j,{\bar b}_j)
\preccurlyeq({\bar a}_1,{\bar b}_1)}
\!\!\!\!\!\!k_2k_j
\begin{vmatrix}
{\bar a}_2 & {\bar b}_2
\\
{\bar a}_j & {\bar b}_j
\end{vmatrix}
=\det\biggl(k_2({\bar a}_2,{\bar b}_2),
\!\!\!\!\!\!\sum_{j\colon-({\bar a}_2,{\bar b}_2)\prec({\bar a}_j,{\bar b}_j)
\preccurlyeq({\bar a}_1,{\bar b}_1)}
\!\!\!\!\!\!k_j({\bar a}_j,{\bar b}_j)\biggr).
\end{equation}
Similarly, the sum of terms in~$S_2$ that are not present in~$S_1$
is given by
\begin{equation}
\label{det2}
\sum_{\ell\colon({\bar a}_3,{\bar b}_3)\preccurlyeq({\bar a}_\ell,
{\bar b}_\ell)\prec-({\bar a}_2,{\bar b}_2)}\!\!\!\!\!\!\!\!\!k_2k_\ell
\begin{vmatrix}
{\bar a}_\ell & {\bar b}_\ell
\\
{\bar a}_2 & {\bar b}_2
\end{vmatrix}
=\det\biggl(\sum_{\ell\colon({\bar a}_3,{\bar b}_3)\preccurlyeq
({\bar a}_\ell,{\bar b}_\ell)\prec-({\bar a}_2,{\bar b}_2)}\!\!\!\!\!\!\!\!\!
k_\ell({\bar a}_\ell,{\bar b}_\ell),k_2({\bar a}_2,{\bar b}_2)\biggr).
\end{equation}
The nonconfluency condition $\sum_{i=1}^qk_i({\bar a}_i,{\bar
b}_i)=\sum_{j=1}^m(a_j,b_j)=(0,0)$ implies
that the determinant in the right-hand side of~(\ref{det1}) equals the
determinant in right-hand side of~(\ref{det2}).
This proves that any connected component of the amoeba complement
carries equally many fully supported solutions to the Horn system.

It remains to observe that any solution of a hypergeometric system
with generic parameters can be expanded into a Puiseux series with
the center at the origin. (This series may turn out to be a
Puiseux polynomial.) Since a Puiseux polynomial solution to a Horn
system is defined everywhere except (possibly) the coordinate
hyperplanes, it works for any connected component in the
complement of the amoeba of the singularity. Thus for any such
component~$M$ there exists a Puiseux series basis in the solution
space of the Horn system all of whose elements converge (at least)
in the domain ${\rm Log}^{-1}(M).$

Now we see that we can take pure Puiseux series as a basis. For
this purpose we show that suitable linear combinations of the
analytic continuation of a solution
$$
P(x) = \sum_{k=1}^\mu
x_1^{\frac{v_{1k}}{N_1}}x_2^{\frac{v_{2k}}{N_2}} p_k(x_1, x_2)
$$
where $p_k(x),$ $k=1, \ldots, \mu$ are power series that converge
in ${\rm Log}^{-1} (M_i)$  for a fixed $i,$ $N_1, N_2 \in \N,$
$v_{1k},v_{2k} \in \Z.$ It is worthy noticing that $\mu \leq
N_1\cdot N_2.$ The result of an analytic continuation along the
loop turning around $\ell_1$ times around $x_1=0$ and $\ell_2$
times around $x_2=0$  will be
$$
(M_{x_1=0}^{\ell_1}M_{x_2=0}^{\ell_2})_{\ast} P(x) =
\sum_{k=1}^\mu e^{ \left( \frac{\ell_1 v_{1k}}{N_1} + \frac{\ell_2
v_{2k}}{N_2} \right) 2 \pi {\sqrt{-1}}}
x_1^{\frac{v_{1k}}{N_1}}x_2^{\frac{v_{2k}}{N_2}} p_k(x_1, x_2).
$$
To obtain  $x_1^{\frac{v_{1k}}{N_1}}x_2^{\frac{v_{2k}}{N_2}}
p_k(x_1, x_2)$ as a linear combination of
$(M_{x_1=0}^{\ell_1}M_{x_2=0}^{\ell_2})_{\ast} P(x), $ $ 0 \leq
\ell_1 \leq N_1-1,$ $ 0 \leq \ell_2 \leq N_2-1$ it is enough to
consider the inverse to a Vandermonde matrix of size $\mu.$ This
completes the proof of the theorem.
\end{proof}


\section{Maximally reducible monodromy}
\label{sec:maximallyReducibleMonodromy}

In this section we restrict our attention to bivariate Horn
systems. Let~$A$ be an integer $m\times 2$ matrix whose rows sum
up to the zero vector. Such a matrix, together with the vector of
parameters, defines a bivariate nonconfluent hypergeometric system
of equations. It turns out to be convenient to associate with the
matrix~$A$ the convex polygon~$\mathcal{P}$ with integer vertices
such that the outer normals to the sides of~$\mathcal{P}$ are the
rows of~$A.$ We also require that the relative length of a side
of~$\mathcal{P}$ in the integer lattice equals the number of
occurrences of the corresponding (normal) row in the matrix~$A.$
(Observe that the normals to a polygon whose lengths are adjusted
in this way sum up to zero.) According to the Minkowski theorem the polygon~$\mathcal{P}$ satisfying
these conditions is uniquely determined (up to a translation by an
integer vector) by the matrix~$A.$ Conversely, any plane convex
integer polygon~$\mathcal{P}$ defines the matrix $A(\mathcal{P})$
whose rows are the outer normals to its sides (with some of them
possibly repeated). The order of the rows of this matrix is
unimportant since they all lead to the same hypergeometric system
of equations. Thus, together with the vector of parameters~$c$,
such a polygon defines a nonconfluent hypergeometric system of
equations which we denote by ${\rm Horn}(A(\mathcal{P}),c).$ This
has been illustrated by Example~\ref{ex(1,1)(1,-2),(-2,1)}.

The results of Section~\ref{sec:polynomialBases} yield that any Horn system
defined by a matrix whose rows are the vertices of a simplex or a parallelepiped
admits a basis of Puiseux polynomials for suitable values of its parameters.
In particular, the monodromy representation of such a Horn system
(with this very particular choice of parameters) is maximally reducible.

In the paper~\cite{CDS} the authors have posed the problem of describing the
Gelfand-Kapranov-Zelevinsky hypergeometric systems (see~\cite{gkz89}),
whose solution space contains a one-di\-men\-sion\-al subspace with the trivial
action of monodromy on it. (This corresponds to the existence of a rational
solution.) In the present section, we will resolve the closely related problem
of describing the class of Horn hypergeometric systems with maximally reducible
monodromy representations.
Apart from systems with rational bases of solutions, such systems have the
simplest possible monodromy representation since the corresponding
monodromy groups are generated by diagonal matrices.

Recall that a {\it zonotope} is the Minkowski sum of segments. The main result in this
section is the following theorem.

\begin{theorem} \label{thm:zonotopesHavePolBasis}
The monodromy representation of a bivariate nonconfluent hypergeometric system
${\rm Horn}(A(\mathcal{P}),c)$ is maximally reducible for some $c\in\C^n$ if and
only if the polygon~$\mathcal{P}$ is either
1)~a zonotope; or
2)~the Minkowski sum of a triangle~$\triangle$ and an arbitrary number of segments
that are parallel to the sides of~$\triangle.$
\end{theorem}

For instance, the zonotope in Figure~\ref{fig:zonotope} corresponds to the
matrix~(\ref{eq:zonotopeMatrix}) whose rows are the outer normals to
its sides.

Theorem~\ref{thm:zonotopesHavePolBasis} implies that any triangle defines a hypergeometric
system with a maximally reducible monodromy (for a suitable choice of the vector
of parameters). A quadrilateral defines a system with a maximally reducible monodromy
if and only if it is a trapezoid.

We divide the proof of Theorem ~\ref{thm:zonotopesHavePolBasis} into three steps.

We first give a detailed description of a key technical notion
named "half-space cancellation of poles" (Definition ~\ref{halfspace},
Lemma~\ref{cancellation}). Then we prove that each of the conditions 1),2)
is necessary and sufficient for the conclusion of the theorem to
hold (Propositions \ref{prop:zonotopesufficientcondition},~\ref{prop:PuiseuxMaximallyRed}). Finally we establish
the fact that the maximal reducibility of the monodromy is
equivalent to the existence of a Puiseux polynomial basis for a
proper choice of parameters (Corollary ~\ref{cor:MaximallyReduciblePuiseux}) with the aid of
Proposition~\ref{prop:PuiseuxMaximallyRed}.

To prove the necessity and sufficiency of the condition in
Theorem ~\ref{thm:zonotopesHavePolBasis} we will need the following auxiliary technical
notion.

\begin{definition}
\label{halfspace} We will say that the Ore-Sato coefficient $\varphi(s)
=\frac{\prod_{j=1}^a\Gamma(\alpha_j)}{\prod_{i=1}^b\Gamma(\beta_i)}$
admits a \textit{half-space cancellation of poles,} if the poles
of~$\varphi(s)$ lie in the set $\{s\colon\alpha_j(s) =\sigma,\;
\sigma\in\mathbb Z_{\le0},\; \gamma_j\le\sigma\le0\}$ for some
$\gamma_j<0$, $j\in[1,a]$.
\end{definition}

\begin{lemma}
\label{cancellation} The half-space cancellation of poles in the Ore-Sato
coefficient $\varphi(s)
=\frac{\prod_{j=1}^a\Gamma(\alpha_j)}{\prod_{i=1}^b\Gamma(\beta_i)}$
is a necessary condition for the Mellin-Barnes integral
$\operatorname{MB}(\varphi,\mathcal{C})$ to present a set of
Puiseux polynomial solutions for every contour~$\mathcal{C}$,
satisfying the conditions in Theorem~\ref{thm:MBSolutionToHorn}.
\end{lemma}

\begin{example}
\label{e6.4} Consider the function
$$
\varphi(s)=\frac{\Gamma(s_1+s_2-3)\Gamma(-s_2)}
{\Gamma(s_1+1)\Gamma(s_2+2)\Gamma(-s_2+2)}.
$$
Its poles are located on the lines
$\{s\colon-s_2=\sigma,\;\sigma=-1,0,\; s_1\ne-1,-2,\dots\}$. In
this case $\operatorname{MB}(\varphi,\mathcal{C})
=\operatorname{const}\cdot(x_1+1)^2(2x_1-3x_2+2)$, where the
contour~$\mathcal{C}$ is located around the integer lattice points
inside of $\{s\colon s_1+s_2\le3,\; 0\le s_1,\; 0\le s_2\}$.
\end{example}

We now make use of Definition~\ref{halfspace}  and Lemma~\ref{cancellation} to
prove the sufficiency of either or the conditions~1),~2).

\begin{proposition} \label{prop:zonotopesufficientcondition}
For a polygon~$\mathcal{P}$ of type 1) or 2), ${\rm
Horn}(A(\mathcal{P}),c)$ admits a Puiseux polynomial basis for
some parameter $c\in\C^n$ and hence admits a maximally reducible
monodromy representation.
\end{proposition}
\begin{proof}
Let~$A$ be a $m\times 2$ matrix whose rows are the outer normals
to the sides of a zonotope normalised as described in the
beginning of this section. We will first show that there exists
$c\in\C^m$ such that the space of holomorphic solutions to the
hypergeometric system ${\rm Horn}(A,c)$ at a generic point has a
basis that consists of functions of the form $x^\alpha p(x),$
where $\alpha\in\C^n,$ and $p(x)$ is a (Taylor) polynomial. Since
the analytic continuation of such a function along any path is
proportional to itself, this will prove that the monodromy
representation of ${\rm Horn}(A,c)$ is maximally reducible.

Since the matrix~$A$ defines a zonotope, we may without loss of generality assume (possibly
after interchanging some of its rows) that it consists of blocks of the form
{\small $
B_i=\left(
\begin{array}{rr}
 a_i &  b_i \\
-a_i & -b_i \\
\end{array}
\right).
$}
Let us denote by~$k_i$ the number of occurrences of the block~$B_i$ in the matrix~$A$
and let~$l$ denote the number of different blocks.
By Theorem~\ref{thm:DMStheorem} the holonomic rank of the system ${\rm Horn}(A,c)$
equals
$$
r(A)=
\left( \sum_{i=1}^{l} k_i |a_i| \right) \left( \sum_{j=1}^{l} k_j |b_j| \right) -
\sum_{i=1}^{l} k_{i}^{2} |a_i b_i| =
\sum_{\begin{array}{c}
i,j=1 \\
i\neq j
\end{array}}^{l} k_i k_j |a_i b_j|.
$$
We will use induction with respect to~$l$ to show that the
hypergeometric system ${\rm Horn}(A,c)$ admits a Puiseux
polynomial basis in the linear space of its analytic solutions.
For $l=2$ we have a parallelogram which by
Proposition~\ref{prop:parallelepiped} (for $-\alpha_j - \beta_j \in \N\setminus\{0\}$ in  \eqref{parallelepipedSolution} ) defines a system with a
Puiseux polynomial basis in its solution space.

Let the matrix be defined through $B_{l+1}= 
\left(
\begin{array}{rr}
 a_{l+1} &  b_{l+1} \\
-a_{l+1} & -b_{l+1} \\
\end{array}
\right)$. 

Denote by~$A'$ the matrix that is obtained by appending $k_{l+1}$
copies of the block $B_{l+1}$ to the matrix~$A$ and let~$r(A')$
denote the holonomic rank of the associated Horn system. Similarly
to the above, we may without loss of generality assume that
$a_{l+1}\ne0$, $b_{l+1}\ne0$. We may also assume that the vector
$(a_{l+1},b_{l+1})$ is not proportional to $(a_i,b_i)$ for any
$i=1,\dots,l$. For if these two vectors were proportional, adding
the block~$B_{l+1}$ would be equivalent to increasing the
number~$k_i$ of occurrences of the block~$B_{i}$ in the
matrix~$A$.

Observe that appending the block~$B_{l+1}$ to the matrix~$A$ corresponds to
adding the segment $(-b_{l+1},a_{l+1})$ by Minkowski to the polygon that is defined
by the matrix~$A.$
In this case, the amoeba of the singularity of the corresponding
hypergeometric systems sprouts two new tentacles in opposite
directions. This can be seen from~\cite{PST}, Lemma~11 (two-sided
Abel's lemma). By Theorem~\ref{thm:eachComponentHasABasis} the
number of Puiseux series solutions is the same for every connected
component of its complement. We will show that for a suitable
(and, of course, a very specific) choice of the parameters of the
system these series actually turn out to be polynomials.

Under the above assumptions the holonomic rank~$r(A^{'})$ of the hypergeometric
system defined by the matrix~$A^{'}$ and a generic vector of parameters is given by
$$
r(A^{'}) =
\sum_{ \begin{array}{c}
i,j=1 \\
i\neq j
\end{array}}^{l+1} k_i k_j |a_i b_j| =
r(A) + \sum_{i=1}^{l} k_i k_{l+1} |a_i b_{l+1}| + \sum_{j=1}^{l} k_{l+1} k_j |a_{l+1} b_j|=
$$
$$
r(A) +
\sum_{i=1}^{l}
\left(
(k_i |a_i| + k_{l+1} |a_{l+1}|) (k_i |b_i| + k_{l+1} |b_{l+1}|) -
k_{i}^{2} |a_i b_i| - k_{l+1}^{2} |a_{l+1} b_{l+1}|
\right) =
$$
$$
r(A) + \sum_{i=1}^{l} r(k_i B_i, k_{l+1} B_{l+1}),
$$
where $r(k_i B_i, k_{l+1} B_{l+1})$ stands for the holonomic rank of the
parallelepipedal hypergeometric system defined by the matrix obtained by joining
together~$k_i$ copies of the block~$B_i$ and~$k_{l+1}$ copies of the block~$B_{l+1}.$
Using Proposition~\ref{prop:parallelepiped} and Theorem~\ref{thm:eachComponentHasABasis}
we conclude that adding (by Minkowski) a segment to a plane zonotope preserves the property
of the corresponding hypergeometric system to have a Puiseux polynomial basis in its space of
holomorphic solutions (for a suitable choice of the vector of parameters).

We first observe that for some
positive integer $m_{l+1}$ the poles of the meromorphic function
$$
\frac{  \Gamma (a_{l+1}s_1+ b_{l +1}s_2 + c_{l+1})}{\Gamma
(a_{l+1}s_1+ b_{l +1}s_2 + c_{l+1} + m_{l+1} + 1)}
$$
are located on the lines $\bigcup\limits_{h=0}^{m_{l+1}} \{s:
a_{l+1}s_1+ b_{l +1}s_2 + c_{l+1}+ h=0\}.$ 
 The poles of the function
$$
\prod_{i=1}^{l}\prod_{j=1}^{k_i}
\frac{\Gamma(a_{i}s_1+ b_{i}s_2+c_{i,j})}
{\Gamma(a_{i}s_1+b_{i}s_2+c_{i,j}+m_{i,j}+1)}
$$
are also located in the finite family of lines
$\bigcup_{i=1}^{l}\bigcup_{j=1}^{k_i}\bigcup_{h=0}^{m_{i,j}}
\{s\colon a_{i}s_1+b_{i}s_2+c_{i,j}+h=0\}$. We conclude that for a
suitable choice of the vector of parameters~$c$ the number of
double poles of the meromorphic function
$$
\prod_{i=1}^{l+1}\prod_{j=1}^{k_i}
\frac{\Gamma(a_{i}s_1+b_{i}s_2+c_{i,j})}
{\Gamma(a_{i}s_1+b_{i}s_2+c_{i,j}+m_{i,j}+1)}
$$
is finite. To prove this fact it suffices to choose the vector of
parameters~$c$ so that the parallelogram
$$
\Pi(i,j;k,\ell)=\bigcup_{t=0}^{1}\bigcup_{u=0}^{1}
\{s\colon a_{i}s_1+b_{i}s_2+c_{i,j}+tm_{i,j}=0,\;
a_{k}s_1+b_{k}s_2+c_{k,\ell}+um_{k,\ell}=0\}
$$
does not intersect any similar parallelogram
$\Pi(i',j';k',\ell')$, as long as
$|i-i'|+|j-j'|+|k-k'|+|\ell-\ell'|\ne0$. Remark that all double
poles of the meromorphic function
$$
\frac{\Gamma(a_{i}s_1+b_{i}s_2+c_{i,j})\Gamma(a_{k}s_1+b_{k}s_2+c_{k,\ell})}
{\Gamma(a_{i}s_1+b_{i}s_2+c_{i,j}+m_{i,j}+1)\Gamma(a_{k}s_1+b_{k}s_2
+c_{k,\ell}+m_{k,\ell}+1)},
$$
that contribute to the solutions of ${\rm Horn}(A,c)$, are
contained in  $\Pi(i,j;k,\ell)$  the parallelogram defined above thanks to the
cancellation of poles (cf.~Definition~\ref{halfspace}) of the two
factors $\Gamma(a_{i}s_1+b_{i}s_2+c_{i,j})$ and
$\Gamma(a_{k}s_1+b_{k}s_2+c_{k,\ell})$. Since a parallelogram is
the image of the square $\{(t,u)\colon0\le t\le1,\; 0\le u\le1\}$
under a linear map, it is possible to choose values of the
parameters~$c_{i,j}$, $c_{k,\ell}$, $c_{i',j'}$, $c_{k',\ell'}$ so
that $\Pi(i,j;k,\ell)$ does not intersect~$\Pi(i',j';k',\ell')$
for $(i,j;k,\ell)\ne(i',j';k',\ell')$. The set of such pairs is
finite and therefore the desired choice of parameters can always
be made.

The inductive step described above is illustrated by
Figure~\ref{fig:addingSidesToZonotope} under the assumption that
$a_i,b_i > 0$ for $i=1,2,3$. The shaded regions contain the
supports of Puiseux polynomial solutions to the Horn system
obtained by adding the block $B_3=\left(
\begin{array}{rr}
 a_3 &  b_3 \\
-a_3 & -b_3 \\
\end{array}
\right)
$
to the hypergeometric system defined by the matrix composed of the blocks $B_1$ and $B_2.$
The above rank computation shows that the Puiseux polynomial solutions emerging at the intersections
of the new (the third) pair of divisors with the initial divisors is exactly sufficient to
compensate the rank growth.
In fact, by Theorem~\ref{thm:DMStheorem} the rank of the system defined by all three pairs of divisors
equals $(a_1 + a_2 + a_3)(b_1 + b_2 + b_3) - a_1 b_1 - a_2 b_2 - a_3 b_3.$
This is exactly how many Puiseux polynomials are supported by the three parallelograms depicted in Figure~\ref{fig:addingSidesToZonotope}.
\begin{figure}[ht!]
\vbox{
\begin{minipage}{5.7cm}
\begin{picture}(160,160)
  \put(150,10){\line(-1,1){140}}
  \put(160,20){\line(-1,1){140}}
  \put(90,90){\vector(1,1){10}}
  \put(100,93){\text{\tiny $(a_3, b_3)$}}
  \put(130,30){\vector(-1,-1){10}}
  \put(100,12){\text{\tiny $(-a_3, -b_3)$}}
  \put(118,69){\vector(-1,-1){10}}
  \put(120,70){\text{\tiny  $a_2 b_3 + a_3 b_2$ solutions supported here}}
  \put(70,0){\line(-1,4){40}}
  \put(85,0){\line(-1,4){40}}
  \put(60,40){\vector(-4,-1){15}}
  \put(18,28){\text{\tiny $(-a_1, -b_1)$}}
  \put(50,140){\vector(4,1){15}}
  \put(60,150){\text{\tiny $(a_1, b_1)$}}
  \put(67,122){\vector(-1,0){21}}
  \put(70,120){\text{\tiny $a_1 b_3 + a_3 b_1$ solutions supported here}}
  \put(0,85){\line(3,-1){160}}
  \put(0,100){\line(3,-1){160}}
  \put(9,97){\vector(1,3){5}}
  \put(8,115){\text{\tiny $(a_2, b_2)$}}
  \put(15,80){\vector(-1,-3){5}}
  \put(5,55){\text{\tiny $(-a_2, -b_2)$}}
  \put(81,104){\vector(-1,-1){27}}
  \put(80,105){\text{\tiny $a_1 b_2 + a_2 b_1$ solutions supported here}}
  \put(33,147){\circle*{3}}
  \put(40,120){\circle*{3}}
  \put(53,127){\circle*{3}}
  \put(60,100){\circle*{3}}
\setshadegrid span <2pt>
\hshade    120    40    58   147       33   33   /
\hshade    100    60    60  121       40   56   /
  \put(90,70){\circle*{3}}
  \put(113,47){\circle*{3}}
  \put(120,60){\circle*{3}}
  \put(143,37){\circle*{3}}
\hshade    47    113    135   70       90   93   /
\hshade    37    143    143   48       113   135   /
  \put(49,84){\circle*{3}}
  \put(53,67){\circle*{3}}
  \put(65,78){\circle*{3}}
  \put(70,62){\circle*{3}}
\hshade    67    53    70    83         49   54   /
\hshade    62    70    70    68         53   69   /
\end{picture}
\end{minipage}
} 
\caption{\small Adding a segment to a zonotope that defines a Horn
system} \label{fig:addingSidesToZonotope}
\end{figure}
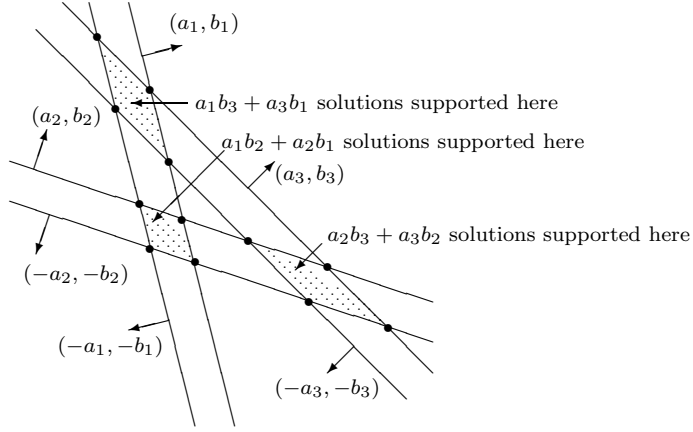

Similar arguments show that the second class of polygons in Theorem~\ref{thm:zonotopesHavePolBasis}
(the sums of triangles in the sense of Minkowski and multiples of their sides) also define\
hypergeometric systems with Puiseux polynomial bases.

Since any pure Puiseux polynomial spans a one-dimensional invariant subspace,
it follows that the monodromy representation of a hypergeometric system satisfying the conditions of
Theorem~\ref{thm:zonotopesHavePolBasis} is maximally reducible.
\end{proof}

Now we prove the necessity of the conditions 1), 2) of
Theorem~\ref{thm:zonotopesHavePolBasis}.

\begin{proposition}
If a hypergeometric system Horn($A,c$) has a maximally reducible
monodromy representation then its Ore-Sato polygon must be either
1) a zonotope or 2) the Minkowski sum of a triangle and segments
parallel to the sides of it. \label{prop:PuiseuxMaximallyRed}
\end{proposition}

\begin{proof}
To simplify the exposition we treat the case where the matrix~$A$
has the following form:
\begin{equation}
\label{eq6.1}
A'=\begin{pmatrix}
1 & 0
\\
0 & 1
\\
a_1 & b_1
\\
\dots & \dots
\\
a_r & b_r
\end{pmatrix},
\end{equation}
where $1+\sum_{j=1}^ra_j=1+\sum_{j=1}^rb_j=0$, $m=r+2$. The proof
for the general form of~$A$ can be achieved in a completely
parallel way.

As a triangle Ore-Sato polygon means condition 2) case, in the
cases that interest us further the number $r$ shall be greater
than 2 so that $m \geq 4.$ Further we shall use the notation
$\alpha_j(s) = a_j s_1 + b_j s_2.$ We consider two groups of
linear functions $\alpha_j(s)$ that are indexed by  $I_+$, $ I_- $
in such a way that  $ j+ \in I_+ $ (resp. $ k- \in I_- $) if and
only if $a_{j+} > 0$ (resp. $a_{k-} < 0$). We then remark that the
poles of $\Gamma(\alpha_{j+}(s)+ \gamma_{j+})$, $
\alpha_{j+}(s)=-m - \gamma_{j+}, m \in \Z_{\geq 0}$ (resp.
$\Gamma(\alpha_{k-}(s)+ \gamma_{k-})$, $  \alpha_{k-}(s)=-m -
\gamma_{k-}, m \in \Z_{\geq 0}$) restricted to the complex plane
$\{s\in\C^2 : s_2+ \delta_2 +n  =0,  n \in \Z_{\geq 0}  \}$ behave
like $s_1 \rightarrow - \infty$ (resp. $s_1 \rightarrow +
\infty$).

For the function
\begin{equation}
\varphi_{2,j+, k-}(s)=
\frac{\Gamma(s_2+ \delta_2)\Gamma(\alpha_{j+}(s)+ \gamma_{j+})\Gamma(\alpha_{k-}(s)+ \gamma_{k-}) }
{\Gamma(1-s_1- \delta_1)\prod_{\ell \not =j+, k-}^r \Gamma(1-\alpha_\ell(s)  - \gamma_\ell)}
\label{eq:phi2jk}
\end{equation}
we  examine the solution subspace  of $S({\rm Horn}(A',c'))$
spanned by
$$
u_{2,j+} (x) = \frac {1}{ (2 \pi \sqrt{-1})^2}
\int\limits_{\mathcal{C}_{2,j+}} \varphi_{2,j+, k-}(s) \, x^{s}
ds,
$$
and its analytic continuations. Here $ c'=( \delta_1,\delta_2,
\gamma_1, \ldots, \gamma_r)$ and
$$ {\mathcal{C}_{2,j+}} = \{s\in\C^2 : | s_2+ \delta_2 +n| =  |\alpha_{j+}(s)+ \gamma_{j+}+m| =\varepsilon,     (n, m) \in \Z_{\geq 0}^2 \}.$$
The circle radius $\varepsilon$ is chosen to be small enough so
that each disk inside the circle contains one isolated double pole
of $\varphi_{2,j+, k-}(s)$.

We remark that the space of solutions to a resonant system
${\operatorname{Horn}}(A',c')$ (see Definition~\ref{def:ResonantHornSystem}) has a
non-diagonalisable monodromy representation except in the trivial
case of a system of holonomic rank~1. That is, for such a system
at least one of the monodromy representation matrices would have a
non-trivial Jordan cell of size at least~2. Thus already it is not
maximally reducible. Therefore we may assume that ${\rm Horn}(A',c')$ is
non-resonant. This means that the solution $u_{2,j+} (x)$ can be
expanded into the Puiseux series
\begin{equation}
\sum_{(n,m) \in \midZ_{\geq 0}^2} c_{n,m} \left(\frac{x_1
^{\frac{b_{j+}}{a_{j+}}}}{x_2} \right) ^{n+\delta_2} x_1^{\frac{-m
- \gamma_{j+}}{a_{j+}}}, \label{eq:u2j}
\end{equation}
in the neighbourhood of  $( \frac{1}{x_1},  \frac{1}{x_2}
)=(0,0).$ Repeated application of the monodromy action
$\frac{1}{x_1} \rightarrow \frac{1}{e ^{2\pi \sqrt{-1}}x_1}$ to
the above series representation of $u_{2,j+} (x)$ produces
$a_{j+}$-dimensional subspace $S_{2,j+} \subset S({\rm
Horn}(A',c'))$ due to the non-degeneracy of a Vandermonde matrix.

Now we consider the analytic continuation of the Puiseux series solution  $u_{2,j+} (x)$  (\ref{eq:u2j}) to
\begin{equation}
u_{2,k-} (x) = \frac {1}{ (2 \pi \sqrt{-1})^2}
\int\limits_{\mathcal{C}_{2,k-}} \varphi_{2,j+, k-}(s) \, x^{s}
ds, \label{eq:u2k}
\end{equation}
by means of the Mellin-Barnes contour throw (See
Fig.~\ref{fig:Mellin-Barnes contour throw}).

\begin{center}
\begin{figure}[ht!]
\vbox{
\begin{minipage}{6cm}
\begin{picture}(150,150)
  \put(0,125){\vector(1,0){20}}
  \put(20,125){\line(1,0){13}}
  \put(40,125){\oval(14,18)[t]}
   \put(40,130){\circle*{2}}
  \put(47,125){\line(1,0){6}}
  \put(60,125){\oval(14,18)[t]}
   \put(60,130){\circle*{2}}
  \put(67,125){\line(1,0){6}}
  \put(80,125){\oval(14,18)[t]}
   \put(80,130){\circle*{2}}
  \put(87,125){\line(1,0){8}}
   \put(100,130){\circle*{2}}
   \put(120,130){\circle*{2}}
   \put(140,130){\circle*{2}}
   \put(10,130){\circle{2}}
   \put(30,130){\circle{2}}
   \put(50,130){\circle{2}}
   \put(70,130){\circle{2}}
   \put(90,130){\circle{2}}
  \put(95,125){\line(0,1){10}}
  \put(95,135){\line(-1,0){8}}
  \put(80,135){\oval(14,18)[t]}
  \put(73,135){\line(-1,0){6}}
  \put(60,135){\oval(14,18)[t]}
  \put(53,135){\line(-1,0){6}}
  \put(40,135){\oval(14,18)[t]}
  \put(33,135){\vector(-1,0){15}}
  \put(20,135){\line(-1,0){20}}
%
   \put(40,20){\circle*{2}}
   \put(60,20){\circle*{2}}
   \put(80,20){\circle*{2}}
   \put(100,20){\circle*{2}}
   \put(120,20){\circle*{2}}
   \put(140,20){\circle*{2}}
   \put(10,20){\circle{2}}
   \put(30,20){\circle{2}}
   \put(50,20){\circle{2}}
   \put(70,20){\circle{2}}
   \put(90,20){\circle{2}}
   \put(35,15){\line(0,1){10}}
   \put(35,15){\line(1,0){8}}
   \put(35,25){\line(1,0){8}}
   \put(50,15){\oval(14,18)[b]}
   \put(50,25){\oval(14,18)[b]}
   \put(57,15){\line(1,0){6}}
   \put(57,25){\line(1,0){6}}
   \put(70,15){\oval(14,18)[b]}
   \put(70,25){\oval(14,18)[b]}
   \put(77,15){\line(1,0){6}}
   \put(77,25){\line(1,0){6}}
   \put(90,15){\oval(14,18)[b]}
   \put(90,25){\oval(14,18)[b]}
   \put(150,15){\vector(-1,0){40}}
   \put(97,25){\vector(1,0){13}}
   \put(97,15){\line(1,0){13}}
   \put(110,25){\line(1,0){40}}
   \put(50,102){\circle{2}}
   \put(55,100){\text{\tiny :\,\, $\alpha_{j+}(s) = 0,-1,-2,\ldots $}}
   \put(50,82){\circle*{2}}
   \put(55,80){\text{\tiny :\,\, $\alpha_{k-}(s) = 0,-1,-2,\ldots $}}
   \put(20,110){\vector(1,-4){18}}
   \put(40,50){\text{\tiny Mellin-Barnes contour throw}}
\end{picture}
\end{minipage}
} 
\caption{Mellin-Barnes contour throw}
\label{fig:Mellin-Barnes
contour throw}
\end{figure}
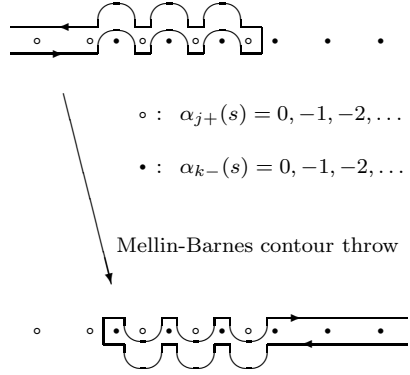
\end{center}

The above integral is calculated as the residue along the contours
$$
{\mathcal{C}_{2,k-}} = \{s\in\C^2 : | s_2+ \delta_2 +n| = |\alpha_{k-}(s)+ \gamma_{k-}+m| = \varepsilon , n,m  \in \Z_{\geq 0} \},
$$
that encircle poles on the
complex plane  $\{s\in\C^2 : s_2 + \delta_2 + n  = 0,  n \in
\Z_{\geq 0} \}$  such that $s_1 \rightarrow + \infty.$ The Puiseux expansion of $ u_{2,k-} (x)$ in the
neighbourhood of $\left( x_1, \frac{1}{x_2} \right)=(0,0)$ has the
following form:
$$
\sum_{(n,m) \in \midZ_{\geq 0}^2} d_{n, m} \left(\frac{x_1
^{\frac{b_{k-}}{a_{k-}}}}{x_2} \right) ^{n+\delta_2} x_1^{\frac{-m
- \gamma_{k-}}{a_{k-}}},
$$
with $a_{k-} <0.$ Repeated application of the monodromy action
$x_1 \rightarrow e ^{2\pi \sqrt{-1}}x_1$ to the above  series
presentation of $ u_{2,k-} (x)$ produces $|a_{k-}|$-dimensional
subspace $S_{2,k-}$ of the solution space $ S({\rm Horn}(A',c'))$ due to
non-degeneracy of a Vandermonde matrix.

Now we analyse the following analytic continuation steps:

a)  The analytic continuation of  $u_{2,j+}$ to $S_{2,k-}$ by Mellin-Barnes contour throw.

b) Monodromy action on $S_{2,k-}$ induced by the map $x_1 \mapsto
e ^{2\pi h \sqrt{-1}} x_1,$ i.e. $\varphi_{2,j+, k-}(s) \, x^{s}
\mapsto \varphi_{2,j+, k-}(s) \, e ^{2\pi h s_1 \sqrt{-1}} x^{s},$
$h \in \Z.$

c) Inverse analytic continuation of $S_{2,k-}$  to $S_{2,j+}.$

Under the condition of the maximal reducibility of monodromy, if
the above procedures a), b), c) give rise to a well-defined
non-trivial monodromy around $x_1 = \infty,$ the image of
$S_{2,j+}$ under this monodromy action has dimension $|a_{k-}|$
and hence $|a_{j+}| =|a_{k-}|.$ This means that for every $ j+ \in
I_{+},$ there exists $ k- \in I_- $ such that $a_{j+}+a_{k-}=0.$

We can apply the same argument in changing the role of $s_2$ and
$s_1,$ i.e.  $x_2$ and $x_1$ in~(\ref{eq:u2j}), (\ref{eq:u2k}) to
conclude that for every $b_{p+} >0$ there exists  $b_{q-} <0$ such
that
 $b_{p+} +  b_{q-} =0.$

Now we show a stronger assertion than the one that has been shown:
for every $j+ \in I_+ ,$ there exists $k- \in I_- $ such that
\begin{equation}
a_{j+} + a_{k-}=0,\;\;\; b_{j+} + b_{k-} = 0.
\label{eq:akbk}
\end{equation}
To prove the existence of such an index, we study the convergence domain of every possible
series defined as a residue of $\varphi_{i,j+, k-}(s) \, x^{s}$.

Let us denote by $D_{j+,k-}$ the convergence domain of the series
$$
u_{j+,k-}(x)=\sum_{n,m \geq 0}
\underset{\tiny \begin{array}{l}\alpha_{j+}(s)+ \gamma_{j+}=-n, \\
\alpha_{k-}(s)+ \gamma_{k-} =-m \end{array}}{\text{Res}}
\varphi_{i,j+, k-}(s) \, x^{s},
$$
for $i=1,2$,  $ j+ \in I_+ ,$ $ k- \in I_- .$
Here we used the notation
$$
\varphi_{1,j+, k-}(s)= \frac{\Gamma(s_1+
\delta_1)\Gamma(\alpha_{j+}(s)+ \gamma_{j+})\Gamma(\alpha_{k-}(s)+
\gamma_{k-}) } {\Gamma(1-s_2- \delta_2)\prod\limits_{\ell \not
=j+, k-}^r \Gamma(1-\alpha_\ell(s)  - \gamma_\ell)}.
$$
In a similar way, we look at the convergence domains $D_{i,j+}$ of
the series
$$
u_{i,j+}(x)=\sum_{n,m \geq 0} \underset{\tiny \begin{array}{c}
\alpha_{j+}(s)+ \gamma_{j+}=-m, \\ s_i +\delta_i=-n
\end{array}}{\text{Res}} \varphi_{i,j+, k-}(s) \, x^{s},
$$
and $D_{i,k-}$ of the  series
$$
u_{i,k-}(x)=\sum_{n,m \geq 0} \underset{\tiny
\begin{array}{c} \alpha_{k-}(s)+ \gamma_{k-}=-m, \\ s_i +\delta_i=-n
\end{array}} {\text{Res}}
\varphi_{i,j+, k-}(s) \, x^{s},
$$
for $i=1,2.$

Now we will establish the following statement: $D_{j+,k-}$ has a
nonempty intersection with at least one of the four domains
$D_{1,j+},$ $D_{2,j+},$ $D_{1,k-},$ $D_{2,k-}.$

To prove this claim we consider the supporting cones $C_{j+,k-},$
$C_{i,j+},$ and $C_{i,k-}$ of the solutions $u_{j+,k-}(x),$ $
u_{i,j+}(x), $ and $u_{i,k-}(x)$ respectively. The Abel lemma
(\cite{gkz89} Proposition 2,~\cite{PST} Lemma 1) implies the
inclusion
$$
{\rm Log}\;x^{(a,b)} - C_{a,b}^{\vee} \subset {\rm
Log}\;(D_{a,b} )
$$
for  some $ x^{(a,b)}  \in D_{a,b}$ and $(a,b) = (j+,k-)$ or
$(i,j+)$ or $ (i,k-).$ After an easy case by case study we see
that  $C_{j+,k-}^{\vee}$ has nonempty two dimensional intersection
with one of four dual cones $C_{1,j+}^{\vee},$ $C_{2,j+}^{\vee},$
$C_{1,k-}^{\vee},$ $C_{2,k-}^{\vee}.$ This proves the claim ( See Figure \ref{fig:nonempty recession cones intersection}).

Let us assume, for example, $D_{j+,k-} \cap D_{2,j+} \not =
\emptyset.$ The analytic continuation of $S_{2,j+}$ induced by a
Mellin-Barnes throw $  \mathcal{C}_{2,j+} \rightarrow
{\mathcal{C}_{j+,k-}}$ on the complex planes $\{s \in \C^2 :
\alpha_{j+}(s)+ \gamma_{j+} \in \Z_{\leq 0} \}$ produces a
$|a_{j+}(b_{j+} + b_{k-})|$-dimensional Puiseux series solution
subspace of $S({\rm Horn}(A',c'))$ convergent on  $D_{j+,k-}$ by
virtue of Theorem \ref{thm:decompositionOfTheSolutionSpace} (2).
This dimension is calculated by the following equalities,
 \begin{equation}
\left|\det \left(
\begin{array}{rr}
a_{j+} &  b_{j+} \\
a_{k-} & b_{k-}
\end{array}
\right) \right|  =  |a_{j+}(b_{j+} + b_{k-})|,
\end{equation}
where $ a_{j+} = -a_{k-}.$ On the other hand, we had already
noticed that the analytic continuation $S_{2,k-}$  of $S_{2,j+}$
induced by the Mellin-Barnes contour throw $\mathcal{C}_{2,j+}
\rightarrow  {\mathcal{C}_{2,k-}}$ on the complex planes $\{s \in
\C^2 : s_2+ \delta_{2} \in \Z_{\leq 0} \}$ has dimension $
|a_{k-}| =a_{j+}.$ Thus we obtained an analytic continuation of
$S_{2,j+}$ convergent on $D_{j+,k-} \cap D_{2,j+} \not =
\emptyset$ with dimension $a_{j+}+|a_{j+}(b_{j+} + b_{k-})|$ by
Theorem \ref{thm:decompositionOfTheSolutionSpace} (2). If the
monodromy is maximally reducible, then every analytic continuation of
$S_{2,j+}$, including the results of monodromy actions, must have
dimension $a_{j+}.$ This means that  $b_{j+} + b_{k-}=0$ and
hence~(\ref{eq:akbk}) follows.

If $D_{j+,k-} \cap D_{2,k-} \not = \emptyset,$ then the same argument as above works.

If $D_{j+,k-} \cap D_{1,j+} \not = \emptyset$ or $D_{j+,k-} \cap
D_{1,k-} \not = \emptyset,$  we interchange the roles of $x_2$ and
$x_1$ and get the equality $|b_{j+}|=|b_{j+}|+| a_{j+} (b_{j+} +
b_{k-})|,$ hence $ b_{j+} + b_{k-}=0.$  Thus again we obtain ~(\ref{eq:akbk}).

If we recall the condition $1+\sum\limits_{j=1}^r a_j = $
$1+\sum\limits_{j=1}^r b_j = 0,$ $m=r+2$ then we see that the matrix $A'$ with
maximally reduced monodromy ${\rm Horn}(A',c')$ must be either
\begin{equation}
\left(
\begin{array}{cc}
1 &  0 \\
0 &  1 \\
-1 &  0 \\
0 &  -1 \\
a_1 &  b_1 \\
-a_1 & -b_1\\
 \vdots & \vdots \\
a_{r/2-1} &  b_{r/2-1} \\
-a_{r/2-1} &  -b_{r/2-1}\\
\end{array}
\right), \;\;\; \; r:{\rm even}.
\label{eq:trapezoid1}
\end{equation}
or
\begin{equation}
\left(
\begin{array}{cc}
1 &  0 \\
0 &  1 \\
-1 &  -1 \\
a_1 &  b_1 \\
-a_1 & -b_1\\
 \vdots & \vdots \\
a_{(r-1)/2} &  b_{(r-1)/2} \\
-a_{(r-1)/2} &  -b_{(r-1)/2}\\
\end{array}
\right), \;\;\;\ r: {\rm odd}.
\label{eq:minkowskisum1}
\end{equation}

Elementary plane geometry shows that the matrix $A'$ like (\ref{eq:trapezoid1}) produces a zonotope Ore-Sato polygon.


To examine the case (\ref{eq:minkowskisum1})  we shall use the notation ${\bf A}_{1-} = (-1,-1), 1- \in I_{-}.$
For $j+ \in I_+$ we see that either $D_{j+, 1-} \cap D_{2,j+} \not = \emptyset$ or $D_{j+, 1-} \cap D_{2,1-} \not = \emptyset$ holds.

If $D_{j+, 1-} \cap D_{2,j+} \not = \emptyset$ the analytic
continuation of the solution
$$
u_{2,j+}(x)=\sum_{n,m \geq 0} \underset{\tiny
\begin{array}{c} \alpha_{j+}(s)+ \gamma_{j+}=-m, \\ s_2 +\delta_2=-n \end{array}}{\text{Res}}
\varphi_{2,1-,j+}(s) \, x^{s},
$$
to
$$
u_{j+, 1-}(x)=\sum_{n,m \geq 0} \underset{\tiny
\begin{array}{c} \alpha_{j+}(s)+ \gamma_{j+}=-m, \\ -s_1-s_2 +\gamma_{1-}=-n
\end{array}}{\text{Res}}
\varphi_{2,1-,j+}(s) \, x^{s},
$$
by Mellin-Barnes contour throw on the complex plane $\{s\in\C^2 :
\alpha_{j+}(s)+ \gamma_{j+}=-m,  m \in \Z_{\geq 0} \}.$ The
argument using
Theorem~\ref{thm:decompositionOfTheSolutionSpace}~(2) would entail
the equality $a_{j+} = a_{j+}+|a_{j+}-b_{j+}|.$ This means that
$a_{j+}-b_{j+}=0.$

If $D_{j+, 1-} \cap D_{j+, 1-} \not = \emptyset,$ the same
argument on the analytic continuation $ u_{2, 1-}(x) \rightarrow
u_{j+, 1-}(x)$ yields the equality $1= 1+|a_{j+}-b_{j+}|.$ Hence
we get $a_{j+}-b_{j+}=0$ again i.e. ${\bf A}_{j+}$ is collinear to
$(-1,-1).$ (See Fig. \ref{fig:nonempty recession cones
intersection}.)

\begin{figure}[ht!]
\begin{center}
\includegraphics[width=8cm]{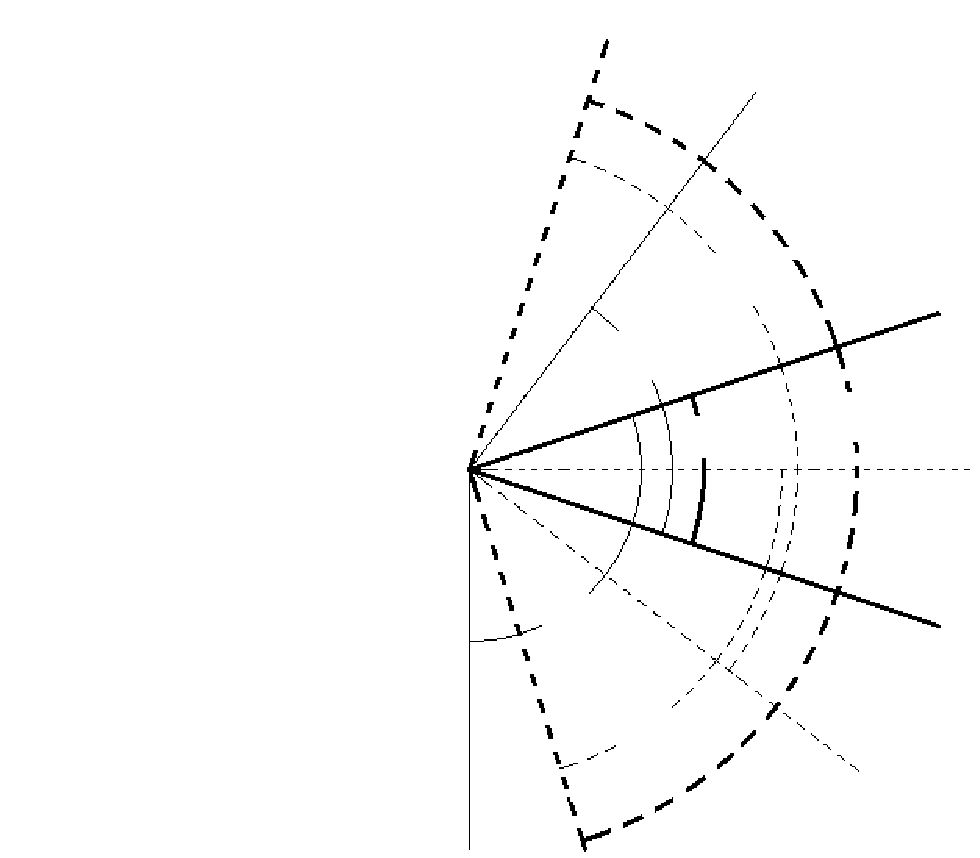}
\vskip-3.85cm \hskip4cm \text{\tiny $C_{j+,1-}$} \vskip3.85cm
\vskip-4.5cm \hskip6.2cm \text{\tiny $C_{j+,1-}^{\vee}$}
\vskip4.5cm
\vskip-6.15cm \hskip4.4cm \text{\tiny $C_{2,1-}^{\vee}$}
\vskip6.15cm
\vskip-6cm \hskip2.9cm \text{\tiny$C_{2,1-}$} \vskip6cm
\vskip-4.4cm \hskip1.7cm \text{\tiny $C_{2,j+}$} \vskip4.4cm
\vskip-3.9cm \hskip2.9cm \text{\tiny $C_{2,j+}^{\vee}$} \vskip1cm
\end{center}
\caption{Recession cones intersection}
    \label{fig:nonempty recession cones intersection}
\end{figure}

In an analogous way we can examine the analytic continuation of
$$
u_{2,1-}(x)=\sum_{n,m \geq 0} \underset{\tiny
\begin{array}{c} -s_1-s_2 +\gamma_{1-}=-m, \\ s_2 +\delta_2=-n
\end{array}}{\text{Res}}
\varphi_{2,1-,j+}(s) \, x^{s},
$$
to
$$
u_{2,k-}(x)=\sum_{n,m \geq 0} \underset{\tiny
\begin{array}{c} \alpha_{k-}(s)+ \gamma_{k-}=-m, \\ s_2 +\delta_2=-n
\end{array}}{\text{Res}}
\varphi_{2,1-,j+}(s) \, x^{s},
$$
by Mellin-Barnes contour throw along the complex planes $\{s \in \C^2 : s_2+ \delta_{2} \in \Z_{\leq 0} \}.$

In view of the relation $C_{2,1+}^{\vee} \subset C_{2,k-}^{\vee},$ we see that $ 1+ |a_{k-}| =1$ i.e. $|a_{k-}| =0$
and ${\bf A}_{k-}$ is collinear to $(0,1).$

We can now apply the same argument to the residues of
$\varphi_{1,1-,j+}(s) \, x^{s}$ and $\varphi_{1,1-,k-}(s) \,
x^{s}.$

In this way we can conclude that every row vector of the
matrix~(\ref{eq:minkowskisum1}) is collinear to one of three
vectors $(1,0), (0,1), (-1,-1).$

This means that the Ore-Sato polygon of the Horn system ${\rm
Horn}(A',c')$ with $A'$ of~(\ref{eq:minkowskisum1}) must be a
Minkowski sum of a triangle and segments parallel to the sides of
it.
\end{proof}

\begin{corollary}
A bivariate hypergeometric system Horn($A,c$) has a maximally reducible
monodromy representation if and only if the solution space of
Horn($A,\tilde{c}$) is spanned by Puiseux polynomials for some
choice of the vector of parameters $\tilde{c}$.
\label{cor:MaximallyReduciblePuiseux}
\end{corollary}

\begin{proof}
If the solution space of the system Horn($A,\tilde{c}$) is spanned by Puiseux polynomials,
evidently its monodromy is maximally reducible.

Proposition~\ref{prop:PuiseuxMaximallyRed} shows that the Ore-Sato
polygon of a hypergeometric system Horn($A,c$) with a maximally
reducible monodromy must be either a zonotope or the Minkowski sum
of a triangle and segments parallel to its sides. After
Proposition~\ref{prop:zonotopesufficientcondition},
Horn($A,\tilde{c}$) admits a Puiseux polynomial basis for a
suitably chosen parameter $\tilde{c}.$
\end{proof}

\begin{example}
\label{ex(1,2)(-1,-2),(-1,1),(1,-1),(-3,-2),(3,2),(2,-1),(-2,1)}
A random zonotope.
\rm
Let us consider the following configuration which is given by
the Minkowski sum of four segments:
\begin{equation}
A=
\left(
\begin{array}{rr}
 1 &  2 \\
-1 & -2 \\
-1 &  1 \\
 1 & -1 \\
-3 & -2 \\
 3 &  2 \\
 2 & -1 \\
-2 &  1
\end{array}
\right).
\label{eq:zonotopeMatrix}
\end{equation}
\begin{center}
\begin{figure}[htbp]
\begin{minipage}{3cm}
\begin{picture}(100,100)
  \put(10,0){\vector(0,1){90}}
  \put(0,10){\vector(1,0){80}}
  \put(50,10){\line(1,1){10}}
  \put(60,20){\line(1,2){10}}
  \put(70,40){\line(-2,3){20}}
  \put(50,70){\line(-2,1){20}}
  \put(30,80){\line(-1,-1){10}}
  \put(20,70){\line(-1,-2){10}}
  \put(10,50){\line(2,-3){20}}
  \put(30,20){\line(2,-1){20}}
  \put(50,10){\circle*{3}}
  \put(30,20){\circle*{3}}
  \put(40,20){\circle*{3}}
  \put(50,20){\circle*{3}}
  \put(60,20){\circle*{3}}
  \put(30,30){\circle*{3}}
  \put(40,30){\circle*{3}}
  \put(50,30){\circle*{3}}
  \put(60,30){\circle*{3}}
  \put(20,40){\circle*{3}}
  \put(30,40){\circle*{3}}
  \put(40,40){\circle*{3}}
  \put(50,40){\circle*{3}}
  \put(60,40){\circle*{3}}
  \put(70,40){\circle*{3}}
  \put(10,50){\circle*{3}}
  \put(20,50){\circle*{3}}
  \put(30,50){\circle*{3}}
  \put(40,50){\circle*{3}}
  \put(50,50){\circle*{3}}
  \put(60,50){\circle*{3}}
  \put(20,60){\circle*{3}}
  \put(30,60){\circle*{3}}
  \put(40,60){\circle*{3}}
  \put(50,60){\circle*{3}}
  \put(20,70){\circle*{3}}
  \put(30,70){\circle*{3}}
  \put(40,70){\circle*{3}}
  \put(50,70){\circle*{3}}
  \put(30,80){\circle*{3}}
\end{picture}
\end{minipage}
\caption{\small The zonotope which defines the
matrix~(\ref{eq:zonotopeMatrix})} \label{fig:zonotope}
\end{figure}
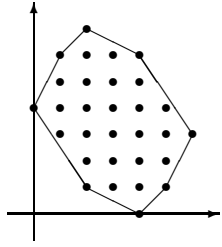
\end{center}
\noindent Choose the vector of parameters to be
$c=(3,-5,-2,1,-2,-1,-1,-1).$ The corresponding hypergeometric
system ${\rm Horn}(A,c)$ is holonomic with rank 31. Here is the
pure Puiseux polynomial basis in its solution space (which was
computed with Mathematica~9.0).
 The persistent solutions are
$x_2,x_1^3
x_2^5,\frac{\sqrt{x_1}}{x_2^{7/4}},\frac{x_1}{x_2^2},\frac{x_1^{5/2}}{x_2^{15/4}},\frac{x_1^3}{x_2^4}$
while non-persistent Puiseux polynomial solutions are
$$
\begin{array}{c}
\frac{x_1^2}{x_2^3}, \, \frac{x_1^{3/2}}{x_2^{11/4}}, \, \frac{1}{x_1^{4/5}x_2^{8/5}}, \, \frac{x_1^{2/7}}{x_2^{3/7}}, \,
\frac{\sqrt[7]{x_2}}{x_1^{3/7}}, \, \frac{x_2^{3/5}}{x_1^{2/5}}, \, \frac{x_2}{x_1}, \,
13068 x_1^2 x_2^4+18900 x_1^2 x_2^3+74529 x_1 x_2^3+715715 x_1 x_2^2,\\
\frac{54}{x_1^{4/7} \sqrt[7]{x_2}}+\frac{5x_1^{3/7}}{\sqrt[7]{x_2}}, \,
\frac{99 x_2^{3/7}}{x_1^{2/7}}-\frac{52}{x_1^{2/7} x_2^{4/7}}, \, \frac{230x_2^{5/7}}{\sqrt[7]{x_1}}-\frac{407}{\sqrt[7]{x_1} x_2^{2/7}}, \,
\frac{5}{x_2}-9,38 \sqrt[7]{x_1} x_2^{2/7}-\frac{99\sqrt[7]{x_1}}{x_2^{5/7}},\\
\frac{234 x_2^{6/5}}{x_1^{4/5}}-\frac{1463 \sqrt[5]{x_2}}{x_1^{4/5}}, \, \frac{14 x_2^{7/5}}{x_1^{3/5}}-\frac{837 x_2^{2/5}}{x_1^{3/5}}, \,
\frac{119 x_2^{4/5}}{\sqrt[5]{x_1}}-\frac{4 x_2^{4/5}}{x_1^{6/5}}, \, \frac{275}{x_1^2 x_2}-\frac{7}{x_1^3x_2}, \,
\frac{129115}{x_1^{5/3} x_2^{2/3}}-\frac{7904}{x_1^{8/3} x_2^{2/3}}, \\
\frac{203}{x_1^{7/3} \sqrt[3]{x_2}}-\frac{170}{x_1^{7/3}x_2^{4/3}}, \, \frac{22869 x_1^2}{x_2^{7/2}}+\frac{16065 x_1^2}{x_2^{5/2}}-\frac{143650 x_1}{x_2^{5/2}}, \,
-\frac{2600150 x_1^{5/2}}{x_2^{13/4}}+\frac{29637333 x_1^{3/2}}{x_2^{9/4}}+\frac{4075291 x_1^{3/2}}{x_2^{13/4}}, \\
\frac{1}{x_1 x_2^2}-\frac{7}{x_1x_2}, \, \frac{19}{x_1^{7/5} x_2^{9/5}}+\frac{143}{x_1^{2/5} x_2^{9/5}}, \,
\frac{238}{x_1^{6/5} x_2^{7/5}}+\frac{999}{\sqrt[5]{x_1}x_2^{7/5}}, \, \frac{2511}{x_1^{3/5} x_2^{6/5}}-\frac{88}{x_1^{3/5} x_2^{11/5}}.
\end{array}
$$
The following picture depicts the supports of the above solutions
to ${\rm Horn}(A,c).$ The big bullets correspond to monomials (both
persistent and not) while the small bullets to all other solutions.
The parallelograms that carry the supports arise as intersections
of the divisors of the defining Ore-Sato coefficient.
\begin{figure}[ht!]
\centering
\includegraphics[width=12cm]{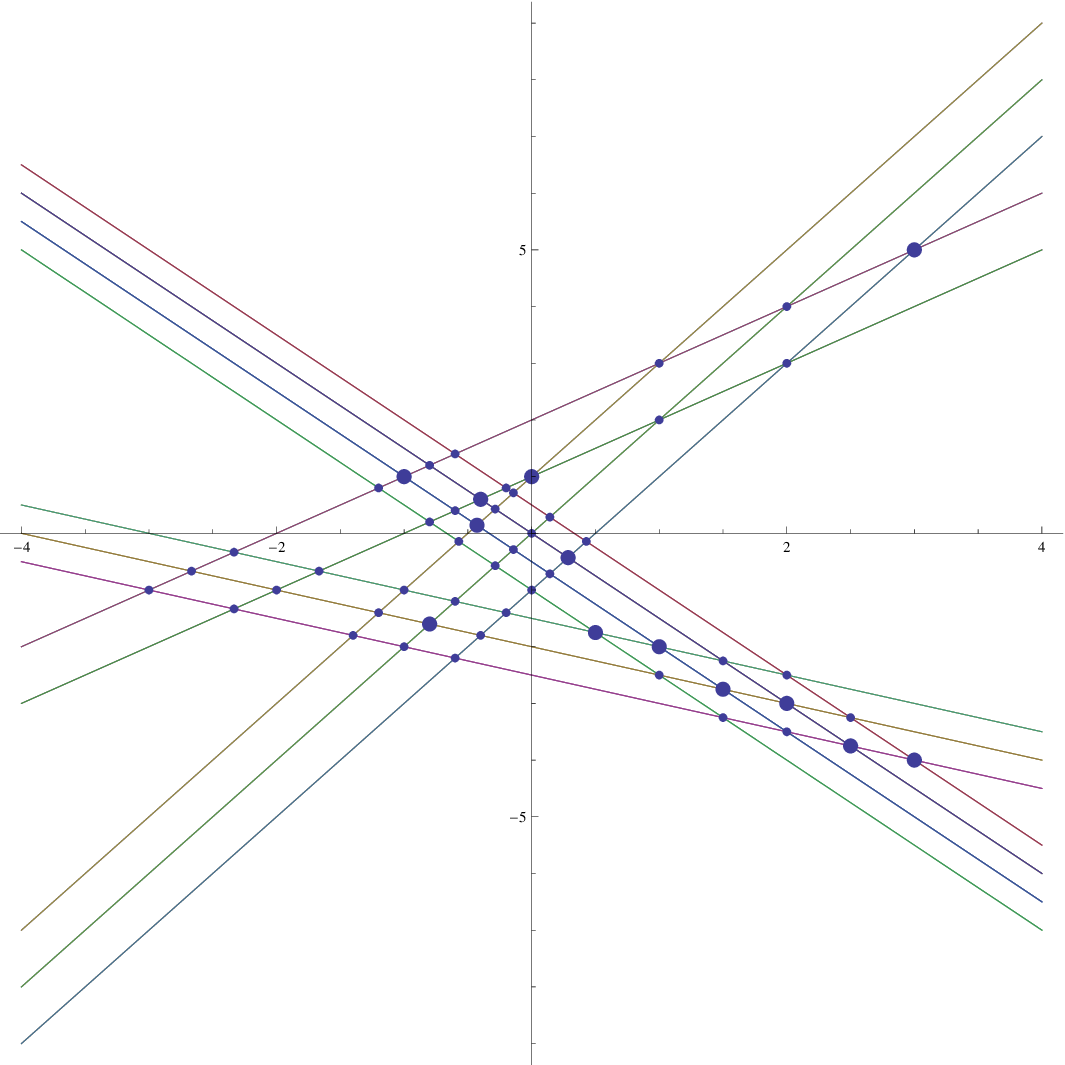}
\caption{\small The supports of the solutions to ${\rm Horn}(A,c)$
defined by (\ref{eq:zonotopeMatrix})} \label{fig:zonotopePicture}
\end{figure}
\end{example}
\begin{example}
\label{ex(2,-1)(2,-1),(-2,1),(-1,3),(-1,3),(1,-3),(1,2),(-1,-2),(-1,-2)}
The sum of a triangle and its sides.
\rm Let us consider the following configuration which is given by
the Minkowski sum of a triangle and all of its sides:
\begin{equation}
A=
\left(
\begin{array}{rr}
 2 & -1 \\
 2 & -1 \\
-2 &  1 \\
-1 &  3 \\
-1 &  3 \\
 1 & -3 \\
 1 &  2 \\
-1 & -2 \\
-1 & -2
\end{array}
\right).
\label{eq:sumOfTriangleAndSidesMatrix}
\end{equation}
\begin{figure}[htbp]
\begin{minipage}{4cm}
\begin{picture}(110,80)
  \put(10,0){\vector(0,1){80}}
  \put(0,10){\vector(1,0){110}}
  \put(50,10){\line(3,1){30}}
  \put(80,20){\line(1,2){10}}
  \put(90,40){\line(1,2){10}}
  \put(100,60){\line(-2,1){20}}
  \put(80,70){\line(-3,-1){30}}
  \put(50,60){\line(-3,-1){30}}
  \put(20,50){\line(-1,-2){10}}
  \put(10,30){\line(2,-1){20}}
  \put(30,20){\line(2,-1){20}}
  \put(50,10){\circle*{3}}
  \put(30,20){\circle*{3}}
  \put(40,20){\circle*{3}}
  \put(50,20){\circle*{3}}
  \put(60,20){\circle*{3}}
  \put(30,30){\circle*{3}}
  \put(40,30){\circle*{3}}
  \put(50,30){\circle*{3}}
  \put(60,30){\circle*{3}}
  \put(20,40){\circle*{3}}
  \put(30,40){\circle*{3}}
  \put(40,40){\circle*{3}}
  \put(50,40){\circle*{3}}
  \put(60,40){\circle*{3}}
  \put(70,40){\circle*{3}}
  \put(10,30){\circle*{3}}
  \put(20,50){\circle*{3}}
  \put(30,50){\circle*{3}}
  \put(40,50){\circle*{3}}
  \put(50,50){\circle*{3}}
  \put(60,50){\circle*{3}}
  \put(50,60){\circle*{3}}
  \put(20,30){\circle*{3}}
  \put(60,60){\circle*{3}}
  \put(70,20){\circle*{3}}
  \put(70,30){\circle*{3}}
  \put(70,50){\circle*{3}}
  \put(70,60){\circle*{3}}
  \put(80,20){\circle*{3}}
  \put(80,30){\circle*{3}}
  \put(80,40){\circle*{3}}
  \put(80,50){\circle*{3}}
  \put(80,60){\circle*{3}}
  \put(80,70){\circle*{3}}
  \put(90,40){\circle*{3}}
  \put(90,50){\circle*{3}}
  \put(90,60){\circle*{3}}
  \put(100,60){\circle*{3}}
\end{picture}
\end{minipage}
{\Large =}
\begin{minipage}{2cm}
\begin{picture}(50,40)
  \put(10,5){\vector(0,1){35}}
  \put(5,10){\vector(1,0){45}}
  \put(30,10){\line(1,2){10}}
  \put(40,30){\line(-3,-1){30}}
  \put(10,20){\line(2,-1){20}}
  \put(10,20){\circle*{3}}
  \put(20,20){\circle*{3}}
  \put(30,10){\circle*{3}}
  \put(30,20){\circle*{3}}
  \put(40,30){\circle*{3}}
\end{picture}
\end{minipage}
{\Large +}
\begin{minipage}{2cm}
\begin{picture}(40,40)
  \put(10,5){\vector(0,1){25}}
  \put(5,10){\vector(1,0){35}}
  \put(10,20){\line(2,-1){20}}
  \put(10,20){\circle*{3}}
  \put(30,10){\circle*{3}}
\end{picture}
\end{minipage}
{\Large +}
\begin{minipage}{2cm}
\begin{picture}(30,40)
  \put(10,5){\vector(0,1){35}}
  \put(5,10){\vector(1,0){25}}
  \put(10,10){\line(1,2){10}}
  \put(10,10){\circle*{3}}
  \put(20,30){\circle*{3}}
\end{picture}
\end{minipage}
{\Large +}
\begin{minipage}{2cm}
\begin{picture}(50,40)
  \put(10,5){\vector(0,1){25}}
  \put(5,10){\vector(1,0){45}}
  \put(10,10){\line(3,1){30}}
  \put(10,10){\circle*{3}}
  \put(40,20){\circle*{3}}
\end{picture}
\end{minipage}
\caption{\small The polygon defining the
matrix~(\ref{eq:sumOfTriangleAndSidesMatrix}) and its Minkowski
decomposition} \label{fig:sumOfTriangleAndSides}
\end{figure}
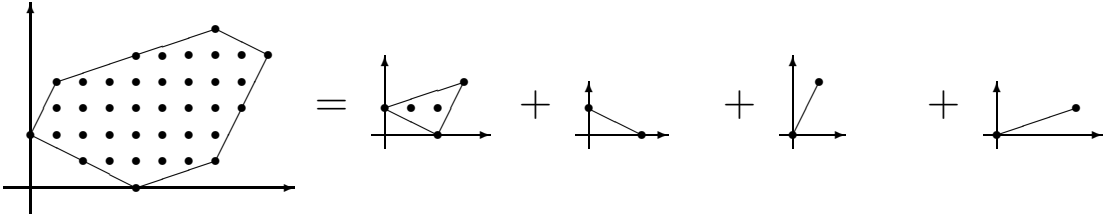
Choose the vector of parameters to be
$c=(-1,-6,3,-2,-10,5,3,-1,-6).$ The corresponding hypergeometric
system is holonomic with rank~$40$ and is defined by the following
differential operators:

$
x_1(\theta_1 - 3\theta_2+ 5) (2 \theta_1 -\theta_2- 6) (2 \theta_1 -\theta_2- 5) (2 \theta_1 -\theta_2- 1)
(2 \theta_1 - \theta_2) (\theta_1 + 2\theta_2+ 3)-
$
$
(\theta_1 + 2\theta_2+ 6) (\theta_1 + 2\theta_2+ 1) (2 \theta_1 -\theta_2- 4) (2 \theta_1 -\theta_2- 3)
(\theta_1 - 3\theta_2+ 10) (\theta_1 - 3\theta_2+ 2),
$

$
x_2 (\theta_1 - 3 \theta_2) (\theta_1 - 3\theta_2+ 1) (\theta_1 - 3\theta_2+ 2) (\theta_1 - 3\theta_2+ 8) (\theta_1 - 3\theta_2+ 9)
$
$
(\theta_1 - 3\theta_2+ 10) (2 \theta_1 -\theta_2- 3) (\theta_1 + 2\theta_2+ 3) (\theta_1 + 2\theta_2+ 4)-
$
$
(\theta_1 - 3\theta_2+ 5) (\theta_1 - 3\theta_2+ 6) (\theta_1 - 3\theta_2+ 7) (2 \theta_1 -\theta_2- 6) (2 \theta_1 -\theta_2- 1)
$
$
(\theta_1 + 2\theta_2) (\theta_1 + 2\theta_2+ 1) (\theta_1 + 2\theta_2+ 5) (\theta_1 + 2\theta_2+ 6).
$

This system has the following five persistent Puiseux polynomial
solutions (which actually turn out to be monomials): $ x_1 x_2,\, x_1^4
x_2^2,\, x_1^{14/5} x_2^{13/5},\, x_1^{13/5} x_2^{21/5},\, x_1^{28/5}
x_2^{26/5}.$ The following thirty pure Puiseux polynomial solutions
to ${\rm Horn}(A,c)$ were computed with Mathematica~9.0: {\small
$$\begin{array} {c}
28 + 15/x_1, \quad x_1^{-4/5}  x_2^{-3/5} (7 x_1 + 22 x_2 + 44 x_1 x_2), \quad x_1^{-1/5} x_2^{-2/5} (196 + 297 x_2 + 231 x_1 x_2),  \\
 x_1^{-3/5} x_2^{-1/5} (198  + 140 x_1 + 165 x_1 x_2), \, x_1^{-7/5} x_2^{1/5} (25 + 120 x_1 + 72  x_1^2), \\
x_1^{4/5} x_2^{-17/5} (3 + 1254 x_2 + 52 x_1  x_2), \\
x_1^{17/5} x_2^{14/5} (298452 + 129675 x_2 + 27930 x_1 x_2 + 588 x_1 x_2^2 + 85 x_1^2 x_2^2), \\
x_1^{2/5} x_2^{-16/5} (91 + 15 x_1 + 15675 x_2 + 3135 x_1 x_2), \quad x_2^{-3} (1040 + 819 x_1 + 62700 x_1 x_2), \\
x_1^{3/5} x_2^{-14/5} (2340 +  182 x_1 + 72675 x_2), \quad x_1^{19/5} x_2^{18/5} (8892 + 266 x_1 + 105 x_2 + 72 x_1 x_2), \\
x_1^3 x_2^3 (426360 + 34884 x_1 + 26600 x_1 x_2 + 1200 x_1^2 x_2 + 51 x_1^2 x_2^2), \\
x_1^{18/5} x_2^{16/5} (43605 + 741 x_1 + 3325 x_2 + 1125 x_1 x_2), \\
x_1^{16/5} x_2^{17/5} (46512 + 6669 x_1 + 900 x_1 x_2 + 64 x_1^2 x_2), \\
{2660 x_1 + 34884 x_1^2 + 51 x_2 + 4500 x_1 x_2 + 74100  x_1^2 x_2}/x_1^7, \\
x_1^{-38/5} x_2^{-1/5} (8151 x_1^2 + 9 x_2 + 1980 x_1 x_2 + 73150 x_1^2 x_2 + 639540  x_1^3 x_2), \\
x_1^{-32/5} x_2^{1/5} (1200 + 33345 x_1 +  170544 x_1^2 + 336 x_2 + 13300 x_1 x_2), \\
x_1^{-34/5} x_2^{2/5} (32 +  1596 x_1 + 17442 x_1^2 + 38760 x_1^3 + 105 x_1 x_2), \\
x_1^{-36/5} x_2^{3/5} (17 +  1575 x_1 + 31122 x_1^2 + 149226 x_1^3), \\ 
x_1^{1/5} x_2^{-18/5} (16 x_1 + 48279 x_2 + 18018  x_1 x_2), \\
x_1^{6/5} x_2^{-8/5} (33 x_1 + 9996 x_2 + 3672 x_1 x_2 + 22100 x_1 x_2^2 + 1326 x_1^2  x_2^2 + 4641 x_1 x_2^3 + 2652 x_1^2 x_2^3), \\
x_1^{9/5} x_2^{-7/5} (81  + 3024 x_2 + 192 x_1 x_2 + 5720 x_2^2 + 1872 x_1 x_2^2 + 624 x_1 x_2^3 + 72 x_1^2  x_2^3), \\
x_1 x_2^{-1} (420 + 216 x_1 + 2925 x_1 x_2 + 175 x_1^2 x_2 +  2145 x_1 x_2^2 + 819 x_1^2 x_2^2), \\
x_1^{8/5} x_2^{-4/5} (23520 + 1728 x_1 +  109200 x_2 + 34125 x_1 x_2 + 38220 x_1 x_2^2 + 2912 x_1^2 x_2^2), \\
x_1^{7/5} x_2^{-6/5} (9504 + 990 x_1 + 128700 x_2 + 41580 x_1 x_2 + 113256 x_1 x_2^2 +  7280 x_1^2 x_2^2 + 4455 x_1^2 x_2^3), \\
x_1^{-22/5} x_2^{-9/5} (1225 x_1^2 + 3780  x_1^3 + 1512 x_1^4 + 75 x_2 + 2730 x_1 x_2 + 18018 x_1^2 x_2 + 27300 x_1^3 x_2),   \\
x_1^{-4} x_2^{-2} (120 x_1^2 + 216 x_1^3 + 45 x_2 + 819 x_1 x_2  + 3250 x_1^2 x_2 + 2925 x_1^3 x_2),                        \\
x_1^{-18/5} x_2^{-11/5} (3456 x_1^2 + 2835 x_1^3  + 5824 x_2 + 65520 x_1 x_2 + 163800 x_1^2 x_2 + 82320 x_1^3 x_2 \\
+ 38220 x_1 x_2^2),    
x_1^{-19/5} x_2^{-13/5} (66 x_1^3 + 2652 x_1 x_2 + 12852 x_1^2 x_2 +  11424 x_1^3 x_2 +\\
 1377 x_2^2 + 18564 x_1 x_2^2 + 48620 x_1^2  x_2^2), 
x_1^{-16/5}  x_2^{-12/5} (198 x_1^2 + 1456 x_2 + 10725 x_1 x_2 +  \\
16632 x_1^2 x_2 + 3696 x_1^3 x_2 + 3432 x_2^2 + 18876 x_1 x_2^2).
\end{array}
$$
} 
We omit the remaining five solutions since they are too cumbersome to display.
Their initial exponents are $(-23/5, 9/5), (-21/5, 8/5), (-19/5, 7/5), (-17/5, 6/5), (-3, 1).$
\end{example}


\end{document}